\documentclass[11pt,reqno, a4wide]{amsart}
\usepackage{amssymb}
\usepackage{amsmath}
\usepackage{amsthm}
\usepackage{mathptmx}
\usepackage{color}
\usepackage[latin1]{inputenc} 

\usepackage[pdftex]{graphicx}
\usepackage{hyperref}
\usepackage{a4wide}

\usepackage[latin1]{inputenc}

\newtheorem{teorema}{Theorem}

\newtheorem{defini}[teorema]{Definition}
\newtheorem{lema}[teorema]{Lemma}

\newtheorem{prop}[teorema]{Proposition}
\newtheorem{obs}[teorema]{Remark}

\numberwithin{equation}{section}
\newcommand{\T}{\mathbb{T}}

\begin{document}

\title[On the quadratic NLS in compact manifolds with boundary]{ Local well--posedness for the  quadratic Schr\"odinger equation  in two--dimensional compact manifolds with boundary}
\author{ Marcelo Nogueira}
\address{Department of Mathematics, State University of Campinas, 13083-859, Campinas, SP, Brazil}
\email{marcelonogueira19@gmail.com }
\thanks{M. Nogueira was supported by CNPq, Brazil.}
\author{Mahendra Panthee}
\address{Department of Mathematics, State University of Campinas, 13083-859, Campinas, SP,  Brazil}
\email{mpanthee@ime.unicamp.br}
\thanks{M. Panthee was partially supported by CNPq (308131/2017-7) and FAPESP (2016/25864-6) Brazil.}


\begin{abstract}
We consider the quadractic NLS posed on a bidimensional compact Riemannian manifold $(M, g)$ with $ \partial M \neq \emptyset$. Using  bilinear and gradient bilinear Strichartz estimates for Schr\"odinger operators in two-dimensional compact manifolds proved by J. Jiang in \cite{JIANG} 
we deduce a new evolution bilinear estimates. Consequently, using Bourgain's spaces,  we obtain a  local well-posedness result for given data $u_0\in H^s(M)$ whenever $s> \frac{2}{3}$ in such manifolds. 
\end{abstract}

\maketitle

\section{Introduction}

 Let $(M,g)$ be a compact Riemannian manifold with boundary of dimension $2$. Denote by $\Delta_{g}$ the Beltrami-Laplace operator with respect to metric $g$ on $M$. We consider the Dirichlet or Neumann problem for the quadratic Schr\"odinger equation  on $M$,
\begin{equation}\label{SDS1}
\begin{cases}
i \partial_{t} u+ \Delta_{g} u =    \alpha  u^{2} + \beta  \overline{u}^2 + \gamma |u|^2,  \qquad{ }   \mbox{in }  [0, \infty) \times (M \setminus \partial M) \\
u(0,x) = u_{0}(x), \\
B u(t, x) = 0, \qquad \mbox{ on } \partial M,
\end{cases}
\end{equation}
where $u = u(t,x)$ is a complex function, $\alpha , \beta , \gamma \in \mathbb{C}$ are complex constants, $B$ is the boundary operator, either $B f = f \mid_{\partial M}$ in the Dirichlet case
or $B f = \partial_{\nu} f \mid_{\partial M}$ in the Neumann case, with $\partial_{\nu}$ denoting the  normal derivative\footnote{We have $\partial_{\nu} u = \frac{\partial u}{\partial \nu} = \langle \nabla u, \nu\rangle$ (Normal component of $\nabla u$), where $\nu$ is the unit outward-pointing normal to $\partial M$.} along the boundary  $\partial M$.

The study of well-posedness issues to initial value problem (IVP) associated with the  nonlinear Schr\"odinger (NLS) equation with quadratic nonlinearities has attracted  attention of several mathematicians in the last decades (see for example \cite {BT2006,KPV96,KIS08} and references therein).  In these works such issues are addressed considering Euclidean spatial domains.  Very few is known when the problem is posed on general manifolds. In this work, we are interested  in addressing the well-posedness issues for the IVP \eqref{SDS1} when $M$ is a compact Riemannian manifold with boundary.

As in the Euclidean case, Strichartz's type inequalities play a vital role while dealing with the well-posedness issues for given data with low Sobolev regularity. In the case of general manifold global in time Strichartz's type inequalities are not available.  In a pioneer work \cite{B1993}, Bourgain obtained local in time Strichartz's type estimate for the NLS equation posed on a standard flat 2-torus $\T^2$. In recent time, much attention has been drawn to find such estimates  with loss of derivative  for  solutions of the linear Schr\"odinger equation
\begin{equation}\label{lin-NLS}
i \partial_{t} u + \Delta_{g} u = 0 ;  \mbox{ } u(x, 0) = f(x),
\end{equation}
 posed on general manifolds of dimension $d \geq 2$. In  such spaces, these estimates are given by 
 \begin{equation}\label{Strichartz}
  \|e^{it \Delta} f \|_{L^{p}(I; L^{q}(M))} \leq C(d, p, q, I) \|f\|_{H^{\ell}(M)},   
 \end{equation}
 for some  $0 \leq \ell < \frac{2}{p}$ where $(p,q)$ is a $d$- admissible pair, i.e., $2/p + d/q = d/2$ with $q < \infty$ and $I \subset \mathbb{R}$ is a finite interval. The number $\ell$ depends on the geometry of $M$ and is called the loss of regularity index. In the flat case, $M = (\mathbb{R}^d, \delta_{ij})$ we have $\ell = 0$ and one can take $I = \mathbb{R}$. For a complete discussion we refer the readers to \cite{LP2015B,TAO2006B} and references therein. 
 
  For a pioneer work concerning the  Strichartz's type  estimate  considering non-Euclidean geometries, we refer to  \cite{BGT}, where a local in time version of \eqref{Strichartz} with loss of derivative ($\ell = 1/ p$) for the solution of \eqref{lin-NLS} posed on compact manifold without boundary was derived. Such estimates for the NLS equation posed on manifolds with boundary can be found in \cite{BSS2008} and \cite{BSS2012}.

A powerful refinement of the estimate \eqref{Strichartz} is known as bilinear Strichartz's type estimate
\begin{equation}\label{BilinearEstimates}
\left(\int_{[0,1] \times M} |e^{it \Delta} f(x) \mbox{ } e^{it \Delta} h(x)|^{2} dt dx\right)^{\frac{1}{2}} \leq C (\min (\Gamma, \Lambda))^{s} \|f\|_{L^{2}} \|h\|_{L^{2}},
\end{equation}
that  hold for  $s > s_{0}(M)$ and spectrally localized  $f$, $h$  in
dyadic intervals of order $\Gamma, \Lambda$ respectively, i.e., 
\[
\textbf{1}_{\Lambda \leq \sqrt{- \Delta} \leq 2 \Lambda}(f) = f, \qquad{ }\textbf{1}_{\Gamma \leq \sqrt{- \Delta} \leq 2 \Gamma}(h) = h.
\]

The claim that the bilinear version \eqref{BilinearEstimates} is a  refinement of \eqref{Strichartz} can be 
justified with the help of the following remark. If $d = 2$, then $(4, 4)$  is a $2$- admissible pair. Considering $h = f$ in \eqref{BilinearEstimates}, we obtain 
\[
\|e^{it \Delta} f \|_{L^4([0,1] \times M)}^2 \lesssim \Lambda^s \| f \|_{L^2(M)}^2,
\]
and consequently
\[
\|e^{it \Delta} f \|_{L^4([0,1] \times M)} \lesssim  \| f \|_{H^{s/2}(M)}.
\]
The estimates \eqref{BilinearEstimates} 
has proven to be one of the most important tools  to obtain the local well-posedness results. More precisely, taking $M = \mathbb{S}^{2}$  endowed with its standard metric, using the precise knowledge of its spectrum $\{\lambda _{k} = k (k + 1) \}_{k \in \mathbb{N}}$  and estimates about spectral projectors of the form
\[
\chi_{\lambda} f =  \sum_{k} \chi(\lambda_{k} - \lambda) P_{k}f,
\]
acting on functions $f$ over $M$, where $\chi \in C^{\infty}_{0}(\mathbb{R})$, Burq, G\'erard and Tzvetkov  in \cite{BGT2},   showed that \eqref{BilinearEstimates}  is true for every $s > s_{0}(M) := \frac{1}{4}$. The authors in \cite{BGT2} also proved the validity of \eqref{BilinearEstimates}  for  bidimensional Zoll manifolds in same range of $s$.

In the case of manifolds with boundary,  $\partial M \neq \emptyset$,  we do not have the precise knowledge of the eigenvalues as  in the cases of 
flat torus and sphere, where we know eigenvalues of the Laplacean precisely. In these cases it is possible  to use  the arithmetic
property of these eigenvalues. For general manifolds with boundary, our poor knowledge of spectrum does not allow us to use the same technique.

 Recently,  Anton \cite{RAMONA} considered manifolds with boundary, $M= \mathbb{B}^{3}$ (the three dimensional ball)  and proved  \eqref{BilinearEstimates} and the following estimate involving gradient
\begin{equation}\label{GradientBilinearEstimate}
   \| (\nabla_{x} ( e^{i t \Delta}f) ) \mbox{ } e^{i t \Delta} h \|_{L^{2}([0,1] \times M)} \leq C\, \Lambda\,(\min (\Lambda, \Gamma))^{s} \| f \|_{L^{2}(M)} \| h \|_{L^{2}(M)}
\end{equation}
for  $s_{0}(\mathbb{B}^{3}) = \frac{1}{2}$. Using  these bilinear estimates, the authors in  \cite{RAMONA}, \cite{BGT2}, obtained local and  well-posedness results for the nonlinear cubic Schr\"odinger equations for initial data in $H^{s}(M)$ for $s > s_{0}(M)$ on such manifolds. Observe that the author  in  \cite{RAMONA} proved the local well-posedness result for the cubic nonlinear  Schr\"odinger equation with Dirichlet boundary condition and radial data in $H^{s}$ for every $s > \frac{1}{2}$. Later, Jiang in \cite{JIANG} considered two dimensional compact manifold with boundary and showed  validity of the  estimates  \eqref{BilinearEstimates} and  \eqref{GradientBilinearEstimate}  for $s > s_{0}(M^{2}) = \frac{2}{3}$ and consequently  obtained local well-posedness of the cubic NLS in $H^s(M)$, $s>\frac23$. 

In this work, we use the techniques used in \cite{RAMONA} and \cite{JIANG} to get crucial bilinear estimates corresponding to the quadratic NLS \eqref{SDS1} posed on a two dimensional compact manifold with boundary and prove the following local well-posedness result.

\begin{teorema}\label{local-Th}
Let $(M, g)$ be a two dimensional compact manifold with boundary. For any $u_{0} \in H^{s}(M) $, with
 $s>s_{0}:= \frac{2}{3}$, there exist $T = T(\|u_0\|_{H^{s}(M)}) > 0$
and a unique solution $u(t)$ of the initial value problem \eqref{SDS1}, on the time interval $[0,T]$, such that
\begin{enumerate}
\item[$(i)$] $ u \in X^{s,b}(M);$
\item[$(ii)$] $ u \in C([0, T], H^{s}(M);$
\end{enumerate}
 for suitable $b$ close to $\frac{1}{2}+$. Moreover the map $u_{0} \mapsto u(t)$ is locally Lipschitz
 from $ H^{s}(M)$ into $C([0, T], H^{s}(M))$.
\end{teorema}

Having proved the local well-posedness of the IVP \eqref{SDS1} in Theorem \ref{local-Th}, a natural question to ask is about the global well-posedness. Generally, conserved quantities play a vital role to answer such question. Recall that, one of the important properties of solutions of the nonlinear Schr\"odinger equations with nonlinearity of the form $N_{p}(u) := |u|^{p-1} u$ $(p>1)$  is that mass and energy of the solutions are conserved (at least for smooth solutions) by the flow. However, in the case of \eqref{SDS1}, by multiplying the equation by $\overline{u}$, integrating and taking the imaginary part, one can easily conclude that the mass is conserved during evolution of system \eqref{SDS1} if 
\begin{equation}\label{MassCond}
Im ( \alpha u^{2} \overline{u} + \beta \overline{u}^{3} + \gamma |u|^{2} \overline{u} ) = 0.    
\end{equation}
For instance, the condition \eqref{MassCond} is satisfied if we take  $\alpha = \gamma \in \mathbb{R}$ and $\beta = 0$. On the other hand, it is much more difficult to obtain a condition for energy conservation. In fact, by multiplying the equation \eqref{SDS1} by $\overline{\partial_{t}u}$ integrating and taking the real  parts we obtain that 
\begin{equation}\label{EnergyCond}
    - \int_{M} \partial_{t}(|\nabla u|^{2}_{g}) dx = \int_{M} 2 Re [\overline{\partial_{t}u} ( \alpha  u^{2} + \beta  \overline{u}^2 + \gamma |u|^2)] dx.
\end{equation}
Looking at the RHS of \eqref{EnergyCond}, we see 
that it is  nontrivial to rewrite it as a derivative of a function involving $u$.  This is one of the main differences between  the equation \eqref{SDS1} and the NLS equation with nonlinearity $N_{p}(u)$. This fact constitutes  an obstacle in the development of a global well-posedness theory for the equation under investigation in this work. However, in some cases where  one has  non-trivial pertubations of the NLS equation,  it is possible to obtain an \textit{a priori} estimate which leads to global well-posedness results (see for instance \cite{NP2020}). 

Before leaving this section, we record some notations that will be used throughout this work. We write $A \sim B$ if there are constants $c_{1}, c_{2} > 0$ such that
$A \leq c_{1} B$ and $B \leq c_{2} A$. Throughout this work we will denote dyadic numbers
$2^{m}$ for $m \in \mathbb{N}$ by capital letters, e.g. $N = 2^{n}, L = 2^{l}, \ldots$. The letter
$C$ will be used to denote a positive constant that may vary from line to line. We use $\|\cdot\|_{L^{p}}$
to denote $L^{p}(M)$ norm and $(\cdot, \cdot)_{L^{2}}$ to denote the inner product in $L^{2}(M)$. Moreover, we denote by $\langle x \rangle :=  \sqrt{1 + x^{2}}$.

\section{Function Spaces and Preliminary Results}

\subsection{Spectral properties of the Laplace-Beltrami operator}
Consider the following eigenvalues problems when $M$ is compact:
\begin{enumerate}
\item[$i)$] Closed problem $-\Delta_{g} f = \lambda f$ in $M$; $\partial M = \emptyset$;
\item[$ii)$] Dirichlet problem $-\Delta_{g} f = \lambda f$ in $M$;  $f|_{\partial M} = 0$;
\item[$iii)$] Neumann problem $-\Delta_{g} f = \lambda f$ in $M$; $ \partial_{\nu} f|_{\partial M} = 0$. 
\end{enumerate}
In our case, we will deal with $(ii)$ and $(iii)$. We have the following standard result about the spectrum.
\begin{teorema}\label{SpectrumOnManifolds}
Let $M$ be a compact manifold with boundary  $\partial M$ (eventually empty), and consider
one of the above mentioned eigenvalue problems. Then:
\begin{enumerate}
\item[$i)$] The set of eigenvalue consists of an infinite sequence \footnote{Sometimes we denote the eigenvalues by $\lambda_{\ell} = \mu_{\ell}^{2}$. }
\[
0 = \lambda_{0} < \lambda_{1}\leq \lambda_{2}\leq ...\leq \lambda_{\ell} \leq \lambda_{\ell + 1} \leq \ldots \rightarrow + \infty,
\]
where $0$ is not an eigenvalue in the Dirichlet problem.
\item[$ii)$] Each eigenvalue has finite multiplicity and the eigenspaces $\mathcal{E}_{j}:= \{u \mid - \Delta_{g} u = \lambda_{j} u \}$ corresponding to distinct eigenvalues are $L^{2}(M)$-orthogonal;
\item[$iii)$]The direct sum of the eigenspaces $\mathcal{E}_{j}$ is dense in $L^{2}(M)$ for the $L^{2}$-norm topology. Furthemore, each eigenfunction
is $C^{\infty}$- smooth and analytic.
\end{enumerate}
\end{teorema}
\begin{proof}
See \cite{PIERREBERARD}, p.$53$.
\end{proof}

From now on, we will list the eigenvalues of the problems $(ii)$ and $(iii)$  as
\[
(0 \leq) \lambda_{1} \leq \lambda_{2} \leq \lambda_{3} \leq \ldots\rightarrow +\infty,
\]
with each eigenvalue repeat a number of times equal to its
multiplicity. The third assertion in Theorem $\ref{SpectrumOnManifolds}$ shows that the sequence $\{e_{j} \}_{j \in \mathbb{N}}$ is an orthonormal basis
of $L^{2}(M)$. For any $f \in L^{2}(M)$, one can write $f = \sum_{j} (f \mid e_{j})_{L^{2}} e_{j}$ in $L^{2}$- sense. We finish this subsection by introducing a spectral projector operator. For a dyadic number $\Gamma$, we use $\textbf{1}_{\Lambda \leq \sqrt{- \Delta} \leq 2 \Lambda}$ to denote the spectral projector
$$\sum_{\Lambda\leq\mu_j\leq 2\Lambda}P_j f=\sum_{\Lambda\leq\mu_j\leq 2\Lambda}e_j\int_Mf\bar{e_j}\,dx,$$
where $\mu_j^2=\lambda_j$ are eigenvalues corresponding to eigenvectors $e_j$ of $-\Delta=-\Delta_g$.

\subsection{Function spaces}
\begin{defini}\label{def-HsXs}
Let $(M,g)$ be a compact Riemannian manifold, and consider the Laplace-Beltrami operator $ -\Delta: = - \Delta_{g}$ on $M$. Let $(e_{k})$ be an
$L^{2}$ orthonormal basis formed by the eigenfunctions of $-\Delta$, with eigenvalues $\lambda_{k}:= \mu_{k}^{2}$. Let $P_{k}$ be the orthogonal
projector along $e_{k}$. For $s \geq 0$ we define  the natural Sobolev space  generated by $(1 - \Delta)^{\frac{1}{2}}$,  $H^{s}(M)$, equipped with the following norm\footnote{For $s < 0$ we define $H^{s}(M)$ as the closure of $L^{2}(M)$ under the norm $(\ref{HsNorm})$.}
\begin{equation}\label{HsNorm}
\|u\|_{H^{s}(M)}^{2} :=\|(1 - \Delta)^{\frac{s}{2}} u\|_{L^{2}(M)} =  \sum_{k} \langle \mu_{k}\rangle^{s} \|P_{k} u\|_{L^{2}(M)}^{2}.
\end{equation}

We define the Hilbert spaces  $X^{s,b}(\mathbb{R}\times M)$ as the completion of $C_{0}^{\infty}(\mathbb{R} \times M)$ with respect to the norm
\begin{equation}\label{XsbNorm}
 \|u\|^{2}_{X^{s,b}(\mathbb{R} \times M)} = \sum_{k} \| \langle \tau  + \mu_{k} \rangle^{b} \langle \mu_{k}\rangle^{\frac{s}{2}} \widehat{P_{k} u}(\tau)\|_{L^{2}(\mathbb{R};L^{2}(M))}^{2}
= \| S(-t) u(t, .)\|_{H^{b}(\mathbb R_{\tau}; H^{s}(M))}^{2},
\end{equation}
where $\widehat{P_{k} u}(\tau)$ denotes the Fourier transform of $P_{k}u$ with respect to the time variable.
\end{defini}

\begin{prop}\label{BasicXsb} The following properties are valid 
\begin{enumerate}
\item[(i)] For  $s_{1} \leq s_{2}$ and $b_{1} \leq b_{2}$, one has  $X^{s_{2},b_{2}}(\mathbb{R}\times M) \hookrightarrow X^{s_{1},b_{1}}(\mathbb{R}\times M)$.
\item[(ii)] $X^{0,\frac{1}{6}}(\mathbb{R}\times M) \hookrightarrow L^{3}(\mathbb{R}, L^{2}(M))$. 
\item[(iii)] If  $b > \frac{1}{2}$,  then the inclusion $X^{s,b}(\mathbb{R} \times M) \hookrightarrow C(\mathbb{R}, H^{s}(M))$ holds.
\end{enumerate}
\end{prop}
\noindent
\begin{proof}The part $(i)$ follows directly from the definition of the $X^{s,b}$-norm in \ref{XsbNorm}. The part $(ii)$ follows
from the fact that $u \in X^{s, b}(\mathbb{R} \times M)$, if  and only if,  $S(-t) u (t, \cdot) \in H^{b}(\mathbb{R}, H^{s}(M))$
and from the immersion $H^{1/6}(\mathbb{R}) \hookrightarrow L^{3}(\mathbb{R}) $.
The proof of $(iii)$, is also a consequence of  \eqref{XsbNorm}.
\end{proof}

In order to use a contraction mapping argument to obtain local existence, we need to define local in time version of $X^{s,b}$.

\begin{defini}
For every compact interval $I \subset \mathbb{R}$, we define the restriction space $X^{s,b}(I \times M)$ as the
space of functions $u$ on $I \times M$ that admit extensions to $\mathbb{R} \times M$ in $X^{s, b}(\mathbb{R} \times M)$. The space $X^{s,b}(I \times M)$ is equipped
with the restriction norm
\[
\|u\|_{X^{s,b}(I \times M)} = \inf_{w \in X^{s,b}(\mathbb{R} \times M)} \{ \|w\|_{X^{s,b}(\mathbb{R} \times M)} \mid w = u \mbox{ on } I\}.
\]
\end{defini}

Another property we are going to use frequently refers to the dyadic decompositions and their relation to the norm of the $X^{s,b}$ spaces. More explicitly,  considering $u$, we can decompose with respect to the space variables as
\[
 u = \sum_{N} u_{N} = \sum_{N} \textbf{1}_{\sqrt{- \Delta} \sim N} (u)
\]
where $N$ denotes the sequence of dyadic integers. Using the definition of the operator $\textbf{1}_{\sqrt{- \Delta} \sim N}$ we can establish the norm equivalence relation
\begin{equation}\label{A1}
 \|u\|_{X^{s,b}}^{2} \sim \sum_{N} N^{2s} \|u_{N}\|_{X^{0,b}}^{2} \sim  \sum_{N} \|u_{N}\|_{X^{s,b}}^{2}.   
\end{equation}

In an  analogous manner, we can decompose $u$ with respect to the ``time-space"  variable
\[
u = \sum_{L} u_{L} = \sum_{L} \textbf{1}_{\langle \tau + \mu_{k}\rangle \sim L}(u)
\]
where $L$ denotes the sequence of dyadic integers. Also, using the definition of the  operator $\textbf{1}_{\langle \tau + \mu_{k}\rangle \sim L}$ we can establish the following norm equivalence
\begin{equation}\label{A2}
\|u\|_{X^{0,b}}^{2} \sim \sum_{L} L^{2b} \|u_{L}\|_{L^{2}(\mathbb{R} \times M)}^{2} \sim  \sum_{L} \|u_{L}\|_{X^{0,b}}^{2} .    
\end{equation}

\subsubsection{Linear estimates in the function spaces}
\begin{prop}\label{LinearEstimates1} (Linear estimates in the $X^{s,b}$ spaces).
Let $b,s >0$ and $u_{0} \in H^{s}(M)$. Then
\begin{enumerate}
\item[$(i)$]
\begin{equation}\label{XsbLinearEstimateA}
 \| S(t) u_{0}\|_{X_{T}^{s,b}(\mathbb{R}\times M)} \lesssim T^{\frac{1}{2} - b} \|u_{0}\|_{H^{s}(M)}
\end{equation}
\item[$(ii)$] Let $0 < b' < \frac{1}{2}$ and $0 < b < 1 - b'$. Then for all $F \in X_{T}^{s, -b'}(M)$ ,
\begin{equation}\label{XsbLinearEstimateB}
\left\| \int_{0}^{t} S(t-t')F(t') dt'\right\|_{X_{T}^{s,b}(\mathbb{R}\times M)} \lesssim  T^{1-b-b'}  \|F\|_{X_{T}^{s,-b'}(\mathbb{R}\times M)},
\end{equation}
provided  $0 < T \leq 1$.
\end{enumerate}
\end{prop}
\begin{proof}
For the proof of this proposition we refer to \cite{BGT1}, \cite{JIANG} or \cite{GINIBRE}.
\end{proof}

 From $(\ref{XsbLinearEstimateA})$ we know that $\|S(t) u_{0}\|_{X^{s,b}_{1}(\mathbb{R}\times M)} \leq C \|u_{0}\|_{H^{s}(M)}$, for some $C > 0$. From the definition
of $X^{s,b}_{T}$ spaces we know that $T_{1} < T_{2}$ implies $X_{T_{2}}^{s,b} \subset X_{T_{1}}^{s,b}$. Therefore for $T \leq 1$,

\begin{equation}\label{XsbLinearEstimateAA}
\| S(t) u_{0}\|_{X_{T}^{s,b}(\mathbb{R}\times M)} \leq C \|u_{0}\|_{H^{s}(M)}.
\end{equation}

\subsection{Bilinear Strichartz estimates and applications}

In this subsection we record some estimates obtained in \cite{JIANG} while working on the well-posedness of the cubic NLS equation
\[
i \partial_{t} u + \Delta_{g} u = 0 ;  \mbox{ } u(x, 0) = f(x),
\]
posed on bi-dimensional compact manifolds with boundary. We start with the following lemma.

\begin{lema}\label{LemaA}
Let $(M,g)$ be a two dimensional compact manifold with boundary. If for any $f , h \in L^{2}(M)$ we have
\begin{equation}
 \textbf{1}_{\Lambda \leq \sqrt{- \Delta} \leq 2 \Lambda}(f) = f, \qquad{ }\textbf{1}_{\Gamma \leq \sqrt{- \Delta} \leq 2 \Gamma}(h) = h,
\end{equation}
where $\Lambda$ and $\Gamma$ are dyadic integers, then for any $s > s_{0} = \frac{2}{3}$, there exists  $C > 0$ such that
\begin{equation}\label{A}
\|e^{i t \Delta} f \mbox{ } e^{i t \Delta} h \|_{L^{2}([0,1] \times M)} \leq C (\min (\Gamma,\Lambda))^{s} \| f \|_{L^{2}(M)} \| h \|_{L^{2}(M)}.
\end{equation}
\end{lema}

\begin{proof}
See \cite{JIANG}, p.85.
\end{proof}

\begin{lema}\label{LemaB}
Let $s>s_{0}=\frac23$, and  $\Gamma, \Lambda$ be dyadic integers. The following statements are equivalent:
\begin{enumerate}
\item[$(1)$] For any $f, h \in L^{2}(M)$ satisfying
\[
\textbf{1}_{\Lambda \leq \sqrt{- \Delta} \leq 2 \Lambda} f = f, \qquad{ } \textbf{1}_{\Gamma \leq \sqrt{- \Delta} \leq 2 \Gamma} h = h
\]
one has
\[
\|e^{i t \Delta} f \mbox{ } e^{i t \Delta} h \|_{L^{2}([0,1] \times M)} \leq C (\min (\Gamma,\Lambda))^{s} \| f \|_{L^{2}(M)} \| h \|_{L^{2}(M)}
\]
\item[$(2)$] For any $b > \frac{1}{2}$ and any $f, h \in X^{0,b}(\mathbb{R} \times M)$ satisfying
\[
\textbf{1}_{\Lambda \leq \sqrt{- \Delta} \leq 2 \Lambda} f = f, \qquad{ } \textbf{1}_{\Gamma \leq \sqrt{- \Delta} \leq 2 \Gamma} h = h
\]
one has
\begin{equation}\label{B}
    \|fh\|_{L^{2}(\mathbb{R} \times M)} \leq C (\min (\Gamma, \Lambda))^{s} \|f\|_{X^{0,b}(\mathbb{R} \times M)} \|h\|_{X^{0,b}(\mathbb{R} \times M)}
\end{equation}
\end{enumerate}
\end{lema}

\begin{proof}
See \cite{JIANG}, p.99.
\end{proof}

\begin{lema}\label{LemaC}
Let $(M,g)$ be a two dimensional compact manifold with boundary. If for any $f , h \in L^{2}(M)$ we have
\[
\textbf{1}_{\Lambda \leq \sqrt{- \Delta} \leq 2 \Lambda}(f) = f, \qquad{ }\textbf{1}_{\Gamma \leq \sqrt{- \Delta} \leq 2 \Gamma}(h) = h.
\]
Then for any $s > s_{0} = \frac{2}{3}$, there exists  $C > 0$ such that

\begin{equation}\label{C}
 \| (\nabla_{x} ( e^{i t \Delta}f) ) \mbox{ } e^{i t \Delta} h \|_{L^{2}([0,1] \times M)} \leq C \Lambda(\min (\Lambda, \Gamma))^{s} \| f \|_{L^{2}(M)} \| h \|_{L^{2}(M)}
\end{equation}

\end{lema}
\begin{proof}
See \cite{JIANG}, p.86.
\end{proof}

\begin{lema}\label{LemaD}
Let $s>s_{0}=\frac23$, and  $\Gamma, \Lambda$ be dyadic integers.  The following statements are equivalent:
\begin{enumerate}
\item[$(1)$] For any $f, h \in L^{2}(M)$ satisfying
\[
\textbf{1}_{\Lambda \leq \sqrt{- \Delta} \leq 2 \Lambda} (f) = f, \qquad{ } \textbf{1}_{\Gamma \leq \sqrt{- \Delta} \leq 2 \Gamma} (h) = h
\]
one has
\[
\|(\nabla_{x} e^{i t \Delta} f) \mbox{ } e^{i t \Delta} h \|_{L^{2}([0,1] \times M)} \leq C \Lambda(\min (\Lambda, \Gamma))^{s} \| f \|_{L^{2}(M)} \| h \|_{L^{2}(M)}
\]
\item[$(2)$] Let $b > \frac{1}{2}$. Then, for any $f, h \in X^{0,b}(\mathbb{R} \times M)$ satisfying
\[
\textbf{1}_{\Lambda \leq \sqrt{- \Delta} \leq 2 \Lambda} (f) = f, \qquad{ } \textbf{1}_{\Gamma \leq \sqrt{- \Delta} \leq 2 \Gamma}(h) = h
\]
one has
\begin{equation}\label{D}
    \|(\nabla_{x} f)h\|_{L^{2}(\mathbb{R} \times M)} \leq C \Lambda (\min (\Lambda, \Gamma))^{s} \|f\|_{X^{0,b}(\mathbb{R} \times M)} \|h\|_{X^{0,b}(\mathbb{R} \times M)}.
\end{equation}

\end{enumerate}
\end{lema}

\begin{proof}
See \cite{JIANG}, p.100.
\end{proof}

\section{Bilinear Estimates}
It is well known that in the framework of Bourgain's
spaces the local well-posedness results can usually be reduced to the proof of adequate
$k$-linear estimates.  In our case, due to the non-linear structure of \eqref{SDS1} it is necessary to prove the bilinear estimates.  These  bilinear estimates are crucial  to perform the contraction argument as we shall see in the next section.

In this section, we use the linear  and the nonlinear estimates stated in the previous section to prove  the following crucial bilinear estimates.

\begin{prop}\label{FirstBilinearEstimate}
Let $s_{0} <  s  $, $b' = \frac{1}{2}-$. Then there exists $C > 0$, such that the bilinear estimate
\begin{equation}\label{FirstBilinearEstimateInequality}
\|u_{1} u_{2}\|_{X^{s, -b'}} \leq C \|u_{1}\|_{X^{s, b}} \|u_{2}\|_{X^{s, b}}
\end{equation}
holds provided  $b = \frac{1}{2} + \varepsilon_{1}$, for an adequate
selection of the parameters $0 <  \varepsilon_{1} << 1$.
\end{prop}

\begin{prop}\label{SecondBilinearEstimate}
Let $s_{0} < s $. Then the bilinear inequality
\begin{equation}\label{SecondBilinearEstimateInequality}
\|u_{1} \overline{u_{2}}\|_{X^{s, -b}} \leq C \|u_{1}\|_{X^{s,b_{2}}} \|u_{2}\|_{X^{s,b_{2}}},
\end{equation}
holds true for some $C > 0$, provided $b= \frac{1}{2} - \varepsilon$, $b_{2} = \frac{1}{2} + \varepsilon_{2}$, for an adequate
selection of the parameters $0 < \varepsilon, \varepsilon_{2} << 1$.
\end{prop}

Before  providing proofs for  Propositions \ref{FirstBilinearEstimate} and \ref{SecondBilinearEstimate}, we record some auxiliary results. We begin with an interpolation lemma. 

\begin{lema}\label{InterpolationLemma}
For every $s' > s_{0} = \frac{2}{3}$, there exist $(b,b')$  satisfying respectively, $0 < b' < \frac{1}{2} < b$ , $b + b' < 1$; and $0 < \varepsilon < 1$, such that
\[
s' > \frac{3\theta}{2} + (s_{0} + \delta) (1- \theta),
\]
\[
b' > \frac{\theta}{6} + (\frac{1}{2} + \varepsilon) (1 - \theta).
\]
\end{lema}
\begin{proof} Let $\frac{2}{3} < s' < 2$ and write $s' = s_{0} + 2 \delta$ for some $(\delta > 0)$. Choose $\theta \in (0,1)$ such that
\[
\frac{3 \theta}{2} + (s_{0} + \delta) (1 - \theta) < s',
\]
then $(\frac{5}{6} - \delta) \theta < s' - (s_{0} + \delta)$, (any choice works if $\frac{5}{6} \leq \delta$, otherwise $\theta$ has to be close enough to $0$). Next, we choose $\varepsilon$ such that $b' > \frac{\theta}{6} + (1 - \theta) b$ with $b' = \frac{1}{2} - 2 \varepsilon$, $b = \frac{1}{2} + \varepsilon$. That happens when
\[
(\frac{1}{2} - 2 \varepsilon) > \frac{\theta}{6} + (1 - \theta) (\frac{1}{2} + \varepsilon) \Longleftrightarrow \varepsilon < \frac{\theta}{ 9 - 3 \theta}.
\]
\end{proof}

To deal with dyadic summations, the following lemma proved in  \cite{BGT1}, page 282,  will be very useful.   
\begin{lema}\label{DyadicSummation}
For every $\gamma > 0$, every $\theta > 0$ there exists $C > 0$ such that if $(c_{N})$ and $(d_{N'})$
are two sequences of non-negative numbers indexed by the dyadic integers, then,
\begin{equation}\label{DyadicInequality}
\sum_{N \leq \gamma N'} \left(\frac{N}{N'}\right)^{\theta} c_{N} d_{N'} \leq C \left(\sum_{N} c_{N}^{2}\right)^{\frac{1}{2}}  \left(\sum_{N'} d_{N'}^{2}\right)^{\frac{1}{2}}.
\end{equation}
\end{lema}

Now, we provide proof of the bilinear estimate stated in Proposition \ref{FirstBilinearEstimate}.

\begin{proof}[Proof of  Proposition \ref{FirstBilinearEstimate}]
Using  the  duality relation between $X^{s, -b'}$ and $(X^{s, -b'})^{\ast}\approx X^{-s, b'}$  (see appendix for the proof of this fact) to prove \eqref{FirstBilinearEstimateInequality}, it  suffices to establish the following inequality
\begin{equation}\label{DualityInequalityEquivalence}
|\int_{\mathbb{R} \times M} u_{1} u_{2} \overline{u_{0}} dx dt| \leq C \|u_{1}\|_{X^{s, b_{1}}} \|u_{2}\|_{X^{s, b_{1}}} \|u_{0}\|_{X^{-s, b'}},
\end{equation}
for all $u_{0} \in X^{-s, b'}$. We start inserting the dyadic decompositions on the spatial frequencies of
\[
u_{j} = \sum_{j} u_{j N_{j}}, \; \; (j = 0, 1, 2) 
\]in the left hand side of \ref{DualityInequalityEquivalence}), where 
\[
u_{j N_{j}} := \sum_{k: N_{j} \leq \mu_{k} \leq 2 N_{j}} P_{k} u.
\]
Related to this dyadic decomposition, we have the following equivalences of norms
 \begin{equation}\label{XsbDyadicNormEquivalence}
    \|u_{j}\|_{X^{s,b}}^{2} \sim \sum_{N_{j} } \|u_{j N_{j}}\|_{X^{s,b}}^{2} \sim \sum_{N_{j} } N_{j}^{2s} \|u_{j N_{j}}\|_{X^{0,b}}^{2},
 \end{equation}
 for $j = 0,1, 2$,  which will be very useful in the next steps of the proof.
 
 Let 
 \begin{equation}\label{I}
  I := \int_{\mathbb{R} \times M} u_{1} u_{2} \overline{u_{0}} dx dt.  
 \end{equation}
 Using the triangle inequality, we obtain
\begin{equation}\label{ISplit}
|I| \leq \sum_{N_{0}, N_{1}, N_{2}}  |\int_{\mathbb{R} \times M} u_{1 N_{1}} u_{2N_{2}} \overline{u_{0N_{0}}} dx dt|.
\end{equation}

 Observe that the summation in $(\ref{ISplit})$  is over all triples of dyadic numbers  $(N_{0}, N_{1}, N_{2})$. To simplify the notation, we write $N = (N_{0}, N_{1}, N_{2})$  and

\begin{equation}\label{IN}
I(N) := \int_{\mathbb{R} \times M} u_{1 N_{1}} u_{2N_{2}} \overline{u_{0N_{0}}} dx dt.
\end{equation}

Under these considerations, we split the summation $\sum_{N}|I(N)|$ in the following manner
\begin{equation}\label{FirstSlip}
 \sum_{N} |I(N)| \leq \sum_{N: N_{2} \leq N_{1}} |I(N)| + \sum_{N: N_{1} < N_{2}} |I(N)|.
\end{equation}

Moreover, we split  the summation $\sum_{N: N_{2} \leq N_{1}} |I(N)|$ in two frequency regimes
\begin{equation}\label{SummationN2LeqN1}
 \sum_{N: N_{2} \leq N_{1}} |I(N)| \leq \sum_{N: N_{2} \leq N_{1}, N_{0} \leq C N_{1}}|I(N)| + \sum_{N: N_{2} \leq N_{1}, N_{0} > C N_{1}}|I(N)| =: \Sigma_{1} + \Sigma_{2}.
\end{equation}
Similarly, the summation $\sum_{N: N_{1} < N_{2}} |I(N)| $ is also splitted in two frequency regimes
\begin{equation}\label{SummationN1LeqN2}
  \sum_{N:  N_{1} < N_{2} } |I(N)| \leq \sum_{N: N_{1} < N_{2}, N_{0} \leq C N_{2}}|I(N)| + \sum_{N:  N_{1} < N_{2}, N_{0} > C N_{2}}|I(N)|=: \Sigma_{3} + \Sigma_{4}.
\end{equation}

Therefore, combining $(\ref{ISplit}), (\ref{FirstSlip}), (\ref{SummationN2LeqN1})$ and  $(\ref{SummationN1LeqN2})$ we arrive
at
\begin{equation}\label{IMajorationSigma1234}
|I| \leq \Sigma_{1} + \Sigma_{2} + \Sigma_{3} + \Sigma_{4}.
\end{equation}

In what follows, we estimate  $\Sigma_{1}$,  $\Sigma_{2}$ and $\Sigma_{3}$, $\Sigma_{4}$ in  the four frequency regimes that appear in  \eqref{SummationN2LeqN1}) and \eqref{SummationN1LeqN2} respectively. To  simplify  the notations further, in sequel, for $j = 1,2$ we define $u^{(j)}_{N_{j}} := u_{j N_{j}}$ and  $u^{(0)}_{N_{0}} := \overline{u_{0 N_{0}}}$.

First we will estimate   $\Sigma_{1}$ and  $\Sigma_{2}$ considering the frequency regimes with condition $N_{2} \leq N_{1} $. Using symmetry,  estimates for the terms  $\Sigma_{3}$ and $\Sigma_{4}$  in the frequency regimes with condition $N_{2} > N_{1} $ follow with  analogous arguments.
For this purpose, we need  the following lemma.

\begin{lema}\label{BoundForI1}
Let, $N_{2} \leq N_{1}$ and $I(N)$ be as defined in \eqref{IN}. If \eqref{B} and \eqref{C} hold for $s > s_{0}$, then for all $s' > s_{0}$ there exist
$0 < b',\,b_{1} < \frac{1}{2}$ and  $C > 0$ such that the following estimates hold
\begin{equation}\label{FirstBoundForI1}
   |I(N)| \leq C N_{2}^{s'}  \|u^{(0)}_{N_{0}}\|_{X^{0, b_{1}}} \|u^{(1)}_{N_{1}}\|_{X^{0, b_{1}}} \|u^{(2)}_{N_{2}}\|_{X^{0, b_{1}}},
\end{equation}

\begin{equation}\label{SecondBoundForI1}
|I(N)|\leq C \Big(\frac{N_{1}}{N_{0}}\Big)^{2} N_{2}^{s'} \| u^{(1)}_{N_{1}}\|_{X^{0,b'}} \| u^{(0)}_{N_{0}}\|_{X^{0,b'}} \| u^{(2)}_{N_{2}}\|_{X^{0,b'}}.
\end{equation}
\end{lema}

\begin{proof} Applying  H\"older's inequality in \eqref{IN}, we obtain
\begin{equation}\label{MM1}
\begin{split}
  |I(N)|  & \leq \int_{\mathbb{R} \times M} \prod_{j = 0}^{2} |u^{(j)}_{N_{j}}| dx dt \leq \|u^{(2)}_{N_{2}}\|_{L^{3}(\mathbb{R}, L^{\infty})} \|u^{(0)}_{N_{0}}\|_{L^{3}(\mathbb{R}, L^{2})} \|u^{(1)}_{N_{1}}\|_{L^{3}(\mathbb{R}, L^{2})}.  \\
\end{split}
\end{equation}
Using the Sobolev embedding $H^{\frac{3}{2}}(M) \hookrightarrow L^{\infty}(M)$, $(\ref{MM1})$ yields

\begin{equation}\label{MM2}
 |I(N)|  \leq  C N_{2}^{\frac{3}{2} } \prod_{ j = 0}^{2} \|u^{(j)}_{N_{j}}\|_{L^{3}(\mathbb{R}, L^{2})}.
\end{equation}
Now, using the property $(ii)$ of  Proposition \ref{BasicXsb}   we obtain 
\begin{equation}\label{I1}
 |I(N)|  \leq  C N_{2}^{3/2} \prod_{ j = 0}^{2} \|u^{(j)}_{N_{j}}\|_{X^{0, \frac{1}{6}}} .
\end{equation}

On the other hand, applying the  Cauchy-Schwarz inequality in $(\ref{IN})$, we obtain
\begin{equation}\label{MM3}
|I(N)|  \leq \|u^{(0)}_{N_{0}} u^{(2)}_{N_{2}}\|_{L^{2}(\mathbb{R} \times M)} \|u^{(1)}_{N_{1}}\|_{L^{2}(\mathbb{R} \times M)}.
\end{equation}
Now, applying  the estimate \eqref{B} in Lemma \ref{LemaB} in the estimate  \eqref{MM3}, we find that for $b> \frac{1}{2}$,

\begin{equation}\label{MM4}
|I(N)| \leq C (\min (N_{0}, N_{2}))^{s_{0} + \delta}  \|u^{(0)}_{N_{0}}\|_{X^{0, b}(\mathbb{R} \times M)} \|u^{(2)}_{N_{2}}\|_{X^{0, b}(\mathbb{R} \times M)}\|u^{(1)}_{N_{1}}\|_{L^{2}(\mathbb{R} \times M)},
\end{equation}
where $s_{0} = \frac{2}{3}$ and $0 < \delta << 1$. 

Hence, from \eqref{MM4} we obtain
\begin{equation}\label{I2}
|I(N)| \leq C N_{2}^{s_{0} + \delta} \prod_{j = 0}^{2} \|u^{(j)}_{N_{j}}\|_{X^{0, b}(\mathbb{R} \times M)}.
\end{equation}

 Now, we further  decompose each $u^{(j)}_{N_{j}}$ (for $j = 0,1,2$) according to the sum of $u^{(j)}_{N_{j}}$ localized in each  spacetime frequency interval of the form $L_{j} \leq \langle \tau + \mu_{k}\rangle < 2 L_{j}$  for  $j = 0,1,2$. More explicitly, for $j = 0, 1, 2$, we decompose
\begin{equation}\label{DyadicDecompostionModulationVariableA}
 u^{(j)}_{N_{j}} = \sum_{L_{j}} u^{(j)}_{N_{j}L_{j}}, \qquad{\mathrm{with} }\qquad u^{(j)}_{N_{j}L_{j}} = \textbf{1}_{L_{j} \leq \langle i \partial_{t} + \Delta \rangle < 2 L_{j}} (u^{(j)}_{N_{j}}),
\end{equation}
where $L_{j}$ denote the dyadic integers. Using \eqref{A2} we obtain the following norm
equivalences for a generic function $u$
\begin{equation}\label{X0bNormEquivalenceA}
 \|u\|_{X^{0,b}}^{2} \sim \sum_{L_{k}} L_{k}^{2b} \|u_{L_{k}}\|_{L^{2}(\mathbb{R} \times M)}^{2} \sim \sum_{L_{k}} \|u_{L_{k}}\|_{X^{0,b}}^{2}.
\end{equation}

Observe that, denoting $L = (L_{0}, L_{1}, L_{2})$, we have
\begin{equation}\label{f-m1}
|I(N)| = |\sum _{L} I(N,L)| \leq \sum _{L} |I(N,L)|,
\end{equation}
where
\begin{equation}\label{f-m2}
I(N,L) := \int_{\mathbb{R} \times M} u^{(0)}_{N_{0}L_{0}}  u^{(1)}_{N_{1}L_{1}}  u^{(2)}_{N_{2}L_{2}} dx dt.
\end{equation}

 Next, using the estimates $(\ref{I1})$ and $(\ref{I2})$, which are also true if we replace the functions $u^{(j)}_{N_{j}}$ by the functions $u^{(j)}_{N_{j}L_{j}}$, and then, using \eqref{X0bNormEquivalenceA} with $u^{(j)}_{N_{j}}$ instead of $u$, we obtain
\begin{equation}\label{I12}
 |I(N, L)| \leq C N_{2}^{3/2} ( L_{0} L_{1} L_{2})^{\frac{1}{6}} \prod_{j = 0}^{2} \|u^{(j)}_{N_{j}L_{j}}\|_{L^{2}(\mathbb{R} \times M)},
\end{equation}
and
\begin{equation}\label{I22}
|I(N, L)| \leq C N_{2}^{s_{0} + \delta} (L_{0} L_{1}L_{2})^{b}  \prod_{ j = 0}^{2} \|u^{(j)}_{N_{j}L_{j}}\|_{L^{2}(\mathbb{R} \times M)}.
\end{equation}
Interpolating the inequalities  $(\ref{I12})$ and $(\ref{I22})$, we have for every $\theta \in (0,1)$ that
\begin{equation}\label{InterpolationBetweenI12andI22}
|I(N, L)| \leq C N_{2}^{\frac{3}{2} \theta + (1 - \theta) (s_{0} + \delta)} ( L_{0} L_{1} L_{2})^{ \frac{\theta}{6} + (1-\theta) b}  \prod_{j = 0}^{2} \|u^{(j)}_{N_{j}L_{j}}\|_{L^{2}(\mathbb{R} \times M)}.
\end{equation}

To continue the interpolation argument we  use Lemma \ref{InterpolationLemma}. It follows from Lemma \ref{InterpolationLemma} that for every $s_{0} < s'$ there exists $b'< 1/2$ such that
\begin{equation}\label{INLAfterInterpolation}
|I(N, L)| \leq C N_{2}^{s'} (L_{0} L_{1} L_{2})^{b'}  \prod_{ j = 0}^{2} \|u^{(j)}_{N_{j}L_{j}}\|_{L^{2}(\mathbb{R} \times M)}.
\end{equation}

Using \eqref{INLAfterInterpolation} in \eqref{f-m1}, we obtain that
\[
|I(N)|  \leq \sum_{L} |I(N, L)|  \leq\sum_{L_{0}, L_{1}, L_{2}} C N_{2}^{s'} (L_{0} L_{1} L_{2})^{b'}  \prod_{ j = 0}^{2} \|u^{(j)}_{N_{j}L_{j}}\|_{L^{2}(\mathbb{R} \times M)}.
\]

Now, choosing $b' < b_{1}$, where $b_{1} \in (b', \frac{1}{2})$ and using the norm equivalence \eqref{X0bNormEquivalenceA}, we have
\begin{equation}\label{f-m3}
\begin{split}
 |I(N)|  & \leq C N_{2}^{s'} \sum_{L_{0}, L_{1}, L_{2}} (L_{0} L_{1} L_{2})^{b' - b_{1}} \|u^{(0)}_{N_{0} L_{0}}\|_{X^{0, b_{1}}}
\|u^{(1)}_{N_{1} L_{1}}\|_{X^{0, b_{1}}} \|u^{(2)}_{N_{2} L_{2}}\|_{X^{0, b_{1}}}\\
 & = C N_{2}^{s'} \sum_{L_{0}, L_{1}} (L_{0} L_{1})^{b'-b_{1}} \|u^{(0)}_{N_{0} L_{0}}\|_{X^{0, b_{1}}}
\|u^{(1)}_{N_{1}L_{1}}\|_{X^{0, b_{1}}} \left(\sum_{L_{2}} L_{2}^{b' - b_{1}} \|u^{(2)}_{N_{2}L_{2}}\|_{X^{0, b_{1}}} \right).
\end{split}
\end{equation}

An application of the Cauchy-Schwarz inequality  in the summation involving $L_{2}$ in \eqref{f-m3} yields
\begin{equation}
\begin{split}\label{f-m4}
 |I(N)| & \leq C N_{2}^{s'} \sum_{L_{0}, L_{1}}  \left(\prod_{j = 0}^{1} L_{j}^{b'-b_{1}}   \|u^{(j)}_{N_{j}L_{j}}\|_{X^{0, b_{1}}}\right) \left(\sum_{L_{2}} L_{2}^{2 (b' - b_{1})}\right)^{\frac{1}{2}}  \left( \sum_{L_{2}} \|u^{(2)}_{N_{2}L_{2}}\|_{X^{0, b_{1}}}^{2}\right)^{\frac{1}{2}} \\
& \leq C N_{2}^{s'} \|u_{N_{2}}^{(2)}\|_{X^{0, b_{1}}}\sum_{L_{0}, L_{1}}  \left(\prod_{j = 0}^{1} L_{j}^{b'-b_{1}}   \|u^{(j)}_{N_{j}L_{j}}\|_{X^{0, b_{1}}}\right) .
\end{split}
\end{equation}

In a similar manner,  using  the Cauchy-Schwarz inequality  in the summations in $L_{0}$ and $L_{1}$ (as was done in the summation involving $L_{2}$) in \eqref{f-m4} and applying
 \eqref{X0bNormEquivalenceA}, we obtain the desired estimate \eqref{FirstBoundForI1}
\begin{equation*}
|I(N)| \leq C N_{2}^{s'}  \|u^{(0)}_{N_{0}}\|_{X^{0, b_{1}}} \|u^{(1)}_{N_{1}}\|_{X^{0, b_{1}}} \|u^{(2)}_{N_{2}}\|_{X^{0, b_{1}}}.
\end{equation*}

Now, we move to prove the estimate \eqref{SecondBoundForI1}. For this,  we need to establish a new estimate for $I(N)$. We start noting that, as the functions $u^{(j)}_{N_{j}}$ are localized at frequency $\sim N_{j}$, we have
\[
u^{(j)}_{N_{j}} = \sum_{\mu_{k} \sim N_{j}} c_{k} e_{k},
\]
where $\mu_{k} \sim N_{j}$ means that $\mu_{k} \in [N_{j}, 2 N_{j}]$ and $c_{k}:=(u^{(j)}_{N_{j}}, e_{k})_{L^{2}}$.
Of course, $e_{k}$ are the eigenfunctions with  eigenvalues $\mu_{k}^{2}$ that is, $-\Delta_{g} e_{k} = \mu_{k}^{2} e_{k}$. For these functions,
we define the operators $T$ and $V$ as follows
\begin{equation}\label{TVOperatorsDefinition}
T(u^{(j)}_{N_{j}}) = \sum_{k: \mu_{k} \sim N_{j}} c_{k} \Big(\frac{N_{j}}{\mu_{k}}\Big)^{2} e_{k}, \qquad{  } \mbox{ and }\qquad{} V(u^{(j)}_{N_{j}}) = \sum_{ k: \mu_{k} \sim N_{j}} c_{k} \Big(\frac{\mu_{k}}{N_{j}}\Big)^{2} e_{k}.
\end{equation}

Observe that, for these operators one has the following  norm equivalences
\begin{equation}\label{TVEquivalenceNorm}
 \|T(u^{(j)}_{N_{j}})\|_{X} \sim \|u^{(j)}_{N_{j}}\|_{X} \sim \|V(u^{(j)}_{N_{j}})\|_{X},
\end{equation}
where $X = L^{2}$ or $ X^{0,b}$. 

Fixing  this terminology, we can write
\[
u^{(j)}_{N_{j}} = - \frac{1}{N_{j}^{2}} \Delta_{g} \Big(\sum_{\mu_{k} \sim N_{j}} c_{k} \Big(\frac{N_{j}}{\mu_{k}}\Big)^{2} e_{k}\Big) = - \frac{1}{N_{j}^{2}}  \Delta_{g} (T(u^{(j)}_{N_{j}})).
\]
In particular,
\[
I(N) = \int_{\mathbb{R} \times M} \prod_{j = 0}^{2} u^{(j)}_{N_{j}}   dx dt = - \frac{1}{N_{0}^{2}} \int_{\mathbb{R} \times M} \Delta_{g} (T(u^{(0)}_{N_{0}})) \prod_{j = 1}^{2} u^{(j)}_{N_{j}}   dx dt.
\]

Now, we apply the Green's theorem \footnote{The Green's theorem states that
\[
\int_{M} (u \Delta v - v \Delta u) dx = \int_{\partial M} (u \frac{\partial v}{\partial \nu} - v \frac{\partial u}{\partial \nu}) d \sigma
\]
where $\partial/\partial \nu$ denotes the normal derivative on the boundary and $d \sigma$ is the induced measure on $\partial M$. Since
we are either assuming Dirichlet or Neumann boundary conditions, all boundary integrals vanish.}
\[
\int_{M} h \Delta f  - f \Delta h dx = \int_{\partial M} h \frac{\partial f}{ \partial \nu}  - f \frac{\partial h}{ \partial \nu} d\sigma,
\]
which, due to the boundary conditions (i.e.,
$u_{N_{j}}^{(j)} \mid_{\partial M} = 0$  
or $  \partial_{\nu} u_{N_{j}}^{(j)} \mid_{\partial M} = 0$ ), leads to the identity
\begin{equation}\label{IntegralI}
I(N) = - \frac{1}{N_{0}^{2}} \int_{\mathbb{R} \times M}  T(u^{(0)}_{N_{0}}) \Delta_{g}\left( u^{(1)}_{N_{1}} u^{(2)}_{N_{2}}\right)   dx dt.
\end{equation}
Next, we apply a general formula \footnote{We have that $( \nabla_{g} f, \nabla_{g}h)_{g}$ is the pointwise scalar product with  respect to the metric  $g$ of $\nabla_{g} f$ and $\nabla_{g} h$.}
\begin{equation}\label{LaplacianOfProduct}
\Delta_{g} (f h ) = h \Delta_{g} f + f \Delta_{g} h + 2 (\nabla_{g} f, \nabla_{g} h)_{g},
\end{equation}
in the identity \eqref{IntegralI}, to get
\begin{equation}\label{INBreak}
 I(N) = I_{1}(N) + I_{2}(N) + I_{3}(N),
\end{equation}
where
\begin{equation}\label{ThreeIntegrals}
\begin{cases}
 I_{1}(N) := - \frac{1}{N_{0}^{2}} \int_{\mathbb{R} \times M}  T(u^{(0)}_{N_{0}}) u^{(1)}_{N_{1}} \Delta_{g}( u^{(2)}_{N_{2}}) dx dt,\\
 I_{2}(N) := - \frac{1}{N_{0}^{2}} \int_{\mathbb{R} \times M}  T(u^{(0)}_{N_{0}}) u^{(2)}_{N_{2}} \Delta_{g}( u^{(1)}_{N_{1}}) dx dt, \\
 I_{3}(N) := - \frac{2}{N_{0}^{2}} \int_{\mathbb{R} \times M}  T(u^{(0)}_{N_{0}}) (\nabla u^{(1)}_{N_{1}},\nabla u^{(2)}_{N_{2}})_{g} dx dt.
\end{cases}
\end{equation}

In what follows, we estimate the terms $ I_{1}(N)$,  $I_{2}(N)$ and $ I_{3}(N)$ separately.

\noindent{\bf Estimate for the Term  $I_{1}(N)$.}
From  $(\ref{TVOperatorsDefinition})$, we can deduce that
\begin{equation}\label{DeltaUjNj}
 \Delta_{g} (u^{(j)}_{N_{j}}) = \sum_{\mu_{k} \sim N_{j}} c_{k} \Delta_{g}(e_{k}) = - N_{j}^{2}\sum_{\mu_{k} \sim N_{j}} c_{k}  \Big(\frac{\mu_{k}}{N_{j}}\Big)^{2}e_{k} = - N_{j}^{2} V(u^{(j)}_{N_{j}}).
\end{equation}
Therefore, if $j  = 2$ in  \eqref{DeltaUjNj},
\begin{equation}\label{MM5}
 \Delta_{g}(u^{(2)}_{N_{2}}) = - N_{2}^{2} V(u^{(2)}_{N_{2}}).
\end{equation}
Now, considering the definition of  the term  $I_{1}$ in $(\ref{ThreeIntegrals})$, using $(\ref{MM5})$ and applying the H\"older's inequality,  we find
\begin{equation}\label{MM6}
 |I_{1}(N)| \leq \Big(\frac{N_{2}}{N_{0}}\Big)^{2}  \int_{\mathbb{R}} \|T(u^{(0)}_{N_{0}})\|_{L^{2}} \|V( u^{(2)}_{N_{2}})\|_{L^{2}} \| u^{(1)}_{N_{1}}\|_{L^{\infty}} dt.
\end{equation}

Using the Sobolev embedding $H^{\frac{3}{2}}(M) \hookrightarrow L^{\infty}(M)$ and  norm equivalence $(\ref{TVEquivalenceNorm})$ with $X = L^{2}$,
we deduce from $(\ref{MM6})$ that
\begin{equation}\label{MM7}
\begin{split}
|I_{1}(N)| & \leq \Big(\frac{N_{2}}{N_{0}}\Big)^{2}N_{1}^{\frac{3}{2}} \int_{\mathbb{R}} \|T(u^{(0)}_{N_{0}})\|_{L^{2}} \|V( u^{(2)}_{N_{2}})\|_{L^{2}} \|u^{(1)}_{N_{1}}\|_{L^{2}} dt\\
 &  \leq \Big(\frac{N_{2}}{N_{0}}\Big)^{2}N_{1}^{\frac{3}{2}} \int_{\mathbb{R}} \|u^{(0)}_{N_{0}}\|_{L^{2}} \|u^{(2)}_{N_{2}}\|_{L^{2}} \|u^{(1)}_{N_{1}}\|_{L^{2}} dt.
\end{split}
\end{equation}

Thus, applying H\"older's inequality in \eqref{MM7} and using the property $(ii)$ in Proposition \ref{BasicXsb}, we get
\begin{equation}\label{I1N2N1a}
|I_{1}(N)| \leq \Big(\frac{N_{1}}{N_{0}} \Big)^{2} N_{2}^{\frac{3}{2}}  \|u^{(0)}_{N_{0}}\|_{X^{0, \frac{1}{6}}} \|u^{(1)}_{N_{1}}\|_{X^{0, \frac{1}{6}}}
\|u^{(2)}_{N_{2}}\|_{X^{0, \frac{1}{6}} }.
\end{equation}

On the other hand, using \eqref{MM5} and the Cauchy-Schwarz inequality in the term  $I_{1}$ in \eqref{ThreeIntegrals}, we obtain
\begin{equation}\label{MM8}
\begin{split}
|I_{1}(N)| &\leq C\Big(\frac{N_{2}}{N_{0}}\Big)^{2} \int_{\mathbb{R}} \| u^{(1)}_{N_{1}}\|_{L^{2}( M)} \| T(u^{(0)}_{N_{0}}) V(u^{(2)}_{N_{2}}) \|_{L^{2}( M)} dt   \\
 & \leq C \Big(\frac{N_{2}}{N_{0}}\Big)^{2}  \| u^{(1)}_{N_{1}}\|_{L^{2}(\mathbb{R} \times M)} \| T(u^{(0)}_{N_{0}}) V(u^{(2)}_{N_{2}}) \|_{L^{2}(\mathbb{R} \times M)}.
\end{split}
\end{equation}

Next, using Lemma \ref{LemaB} and the norm equivalence \eqref{TVEquivalenceNorm} with $X = X^{0,b}$, the estimate
\eqref{MM8} yields
\begin{equation}\label{I1N2N1b}
\begin{split}
|I_{1}(N)| &\leq C \Big(\frac{N_{2}}{N_{0}}\Big)^{2} (\min (N_{0}, N_{2}))^{s_{0}+ \delta} \| u^{(1)}_{N_{1}}\|_{X^{0,b}} \|T( u^{(0)}_{N_{0}})\|_{X^{0,b}} \| V(u^{(2)}_{N_{2}})\|_{X^{0,b}} \\
 & \leq C \Big(\frac{N_{1}}{N_{0}} \Big)^{2} N_{2}^{s_{0}+ \delta} \| u^{(1)}_{N_{1}}\|_{X^{0,b}} \|u^{(0)}_{N_{0}}\|_{X^{0,b}} \| u^{(2)}_{N_{2}}\|_{X^{0,b}} .
\end{split}
\end{equation}

Now, we interpolate $(\ref{I1N2N1a})$ and $(\ref{I1N2N1b})$ as in $(\ref{InterpolationBetweenI12andI22})$, and use Lemma $\ref{InterpolationLemma}$, to obtain
\begin{equation}\label{K1}
  |I_{1}(N)| \leq C \Big(\frac{N_{1}}{N_{0}} \Big)^{2} N_{2}^{s'} \| u^{(1)}_{N_{1}}\|_{X^{0,b'}} \| u^{(0)}_{N_{0}}\|_{X^{0,b'}} \| u^{(2)}_{N_{2}}\|_{X^{0,b'}}
\end{equation}
where we have $s_{0} < s'$ and $b' < \frac{1}{2}$

\noindent{\bf Estimate for the Term  $I_{2}(N)$.}
As in the estimate of the term $I_{1}(N)$, we use \eqref{DeltaUjNj} with $j = 1$  that is,  $\Delta_{g}(u^{(1)}_{N_{1}}) = - N_{1}^{2} V(u^{(1)}_{N_{1}})$ and apply H\"older's inequality in the term $I_{2}$ in \eqref{ThreeIntegrals}, to obtain
\begin{equation}\label{MM9}
 |I_{2}(N)| \leq \Big(\frac{N_{1}}{N_{0}}\Big)^{2}  \int_{\mathbb{R}} \|T(u^{(0)}_{N_{0}})\|_{L^{2}} \|V( u^{(1)}_{N_{1}})\|_{L^{2}} \| u^{(2)}_{N_{2}}\|_{L^{\infty}} dt.
\end{equation}

Now, using the Sobolev embedding $H^{\frac{3}{2}}(M) \hookrightarrow L^{\infty}(M)$ and  norm equivalence $(\ref{TVEquivalenceNorm})$ with $X = L^{2}$,
we deduce from $(\ref{MM9})$ that
\begin{equation}\label{MM10}
\begin{split}
|I_{2}(N)| & \leq \Big(\frac{N_{1}}{N_{0}}\Big)^{2}N_{2}^{\frac{3}{2}} \int_{\mathbb{R}} \|T(u^{(0)}_{N_{0}})\|_{L^{2}} \|V( u^{(1)}_{N_{1}})\|_{L^{2}} \|u^{(2)}_{N_{2}}\|_{L^{2}} dt\\
 &  \leq \Big(\frac{N_{1}}{N_{0}} \Big)^{2}N_{2}^{\frac{3}{2}} \int_{\mathbb{R}} \|u^{(0)}_{N_{0}}\|_{L^{2}} \|u^{(1)}_{N_{1}}\|_{L^{2}} \|u^{(2)}_{N_{2}}\|_{L^{2}} dt.
\end{split}
\end{equation}

Thus, applying H\"older's inequality in \eqref{MM10} and using the property $(ii)$ of Proposition \ref{BasicXsb}, we get
\begin{equation}\label{I2N2N1a}
|I_{2}(N)| \leq \Big(\frac{N_{1}}{N_{0}}\Big)^{2} N_{2}^{\frac{3}{2}}  \|u^{(0)}_{N_{0}}\|_{X^{0, \frac{1}{6}}} \|u^{(2)}_{N_{2}}\|_{X^{0, \frac{1}{6}} } \|u^{(1)}_{N_{1}}\|_{X^{0, \frac{1}{6}}}.
\end{equation}

On the other hand, using Cauchy-Schwarz inequality in the term $I_{2}$ in \eqref{ThreeIntegrals} we obtain that
\begin{equation}\label{MM11}
\begin{split}
|I_{2}(N)| &\leq C \Big(\frac{N_{1}}{N_{0}}\Big)^{2} \int_{\mathbb{R}}  \| T(u^{(0)}_{N_{0}}) u^{(2)}_{N_{2}} \|_{L^{2}( M)} \|V( u^{(1)}_{N_{1}})\|_{L^{2}( M)} dt   \\
 & \leq C \Big(\frac{N_{1}}{N_{0}}\Big)^{2}   \| T(u^{(0)}_{N_{0}}) u^{(2)}_{N_{2}} \|_{L^{2}(\mathbb{R} \times M)} \| V(u^{(1)}_{N_{1}})\|_{L^{2}(\mathbb{R} \times M)}.
\end{split}
\end{equation}

Next, using Lemma \ref{LemaB} in the estimate  \eqref{MM11} and  the relation \eqref{TVEquivalenceNorm} with $X = X^{0,b}$, we get
\begin{equation}\label{I2N2N1b}
\begin{split}
|I_{2}(N)| &\leq C \Big(\frac{N_{1}}{N_{0}} \Big)^{2} (\min (N_{0}, N_{2}))^{s_{0}+ \delta}  \|T( u^{(0)}_{N_{0}})\|_{X^{0,b}} \| u^{(2)}_{N_{2}}\|_{X^{0,b}} \| V(u^{(1)}_{N_{1}})\|_{X^{0,b}} \\
 & \leq C \Big(\frac{N_{1}}{N_{0}} \Big)^{2} N_{2}^{s_{0}+ \delta} \| u^{(0)}_{N_{0}}\|_{X^{0,b}} \|u^{(2)}_{N_{2}}\|_{X^{0,b}} \| u^{(1)}_{N_{1}}\|_{X^{0,b}}.
\end{split}
\end{equation}

Now, we can interpolate \eqref{I2N2N1a} with \eqref{I2N2N1b} as in \eqref{InterpolationBetweenI12andI22}, and use the Lemma \ref{InterpolationLemma} to obtain
\begin{equation}\label{K2}
 |I_{2}(N)| \leq C \Big(\frac{N_{1}}{N_{0}} \Big)^{2} N_{2}^{s'} \| u^{(1)}_{N_{1}}\|_{X^{0,b'}} \| u^{(0)}_{N_{0}}\|_{X^{0,b'}} \| u^{(2)}_{N_{2}}\|_{X^{0,b'}},
\end{equation}
where  $s_{0}< s'$ and  $b' < \frac{1}{2}$.

\noindent{\bf Estimate for the Term $I_{3}(N)$.}
To estimate the term $I_{3}(N)$, we use the inequality
\begin{equation}\label{ScalarProductInequality}
 |(\nabla_{g} f, \nabla_{g} h)_{g}| \leq |\nabla_{g} f| |\nabla_{g} h|,
\end{equation}

with $f = u_{N_{1}}^{(1)}$ and $h = u_{N_{2}}^{(2)} $, and  H\"older's inequality in the identity defining $I_{3}$ in $(\ref{ThreeIntegrals})$, to get
\begin{equation}\label{MM12}
|I_{3}(N)| \leq  \frac{C}{N_{0}^{2}} \int_{\mathbb{R}} \|T(u^{(0)}_{N_{0}})\|_{L^{2}}  \| \nabla u^{(1)}_{N_{1}}\|_{L^{2}} \|\nabla u^{(2)}_{N_{2}}\|_{L^{\infty}} dt
\end{equation}

Now, using  Sobolev embedding $H^{\frac{3}{2}}(M)\hookrightarrow L^{\infty}(M)$ and  the norm equivalence $(\ref{TVEquivalenceNorm})$ with $X = L^{2}$, we obtain from $(\ref{MM12})$
\begin{equation}\label{MM13}
|I_{3}(N)|   \leq  C \frac{N_{2}^{\frac{3}{2}}}{N_{0}^{2}} \int_{\mathbb{R}} \|u^{(0)}_{N_{0}}\|_{L^{2}} \| \nabla u^{(1)}_{N_{1}}\|_{L^{2}} \|\nabla u^{(2)}_{N_{2}}\|_{L^{2}} dt.
\end{equation}
Using the inequality $\|\nabla u_{N}\|_{L^{2}} \leq C N \| u_{N}\|_{L^{2}}$ we deduce from $(\ref{MM13})$ that
\begin{equation}\label{MM14}
|I_{3}(N)|  \leq  C \frac{N_{1} N_{2}^{\frac{3}{2} + 1}}{N_{0}^{2}} \int_{\mathbb{R}} \|u^{(0)}_{N_{0}}\|_{L^{2}} \|  u^{(1)}_{N_{1}}\|_{L^{2}} \| u^{(2)}_{N_{2}}\|_{L^{2}} dt.
\end{equation}

Using H\"older inequality in \eqref{MM14} and  applying the property $(ii)$ in Proposition \ref{BasicXsb}, we get
\begin{equation}\label{I3N2N1a}
 |I_{3}(N)| \leq  C \Big(\frac{N_{1}}{N_{0}}\Big)^{2} N_{2}^{\frac{3}{2}} \|u^{(0)}_{N_{0}}\|_{X^{0, \frac{1}{6}}} \|u^{(1)}_{N_{1}}\|_{X^{0, \frac{1}{6}}}
\|u^{(2)}_{N_{2}}\|_{X^{0,\frac{1}{6}}}.
\end{equation}

 On the other hand, using $(\ref{ScalarProductInequality})$ and  the Cauchy-Schwarz inequality in the term $I_{3}$ of $(\ref{ThreeIntegrals})$, we have
\begin{equation}\label{MM15}
|I_{3}(N)| \leq C \frac{1}{N_{0}^{2}} \| \nabla u^{(1)}_{N_{1}}\|_{L^{2}(\mathbb{R} \times M)} \| (\nabla u^{(2)}_{N_{2}}) T(u^{(0)}_{N_{0}})\|_{L^{2}(\mathbb{R} \times M)}.
\end{equation}

 Next, using the inequality $\|\nabla u_{N}\|_{L^{2}} \leq C N \| u_{N}\|_{L^{2}}$ and the  bilinear estimate \eqref{B} in \eqref{MM15}  we obtain
\begin{equation}\label{MM16}
 |I_{3}(N)|  \leq C \frac{N_{1}}{N_{0}^{2}} N_{2} (\min (N_{0}, N_{2}))^{s_{0} + \delta} \| u^{(1)}_{N_{1}}\|_{L^{2}(\mathbb{R} \times M)} \| u^{(2)}_{N_{2}}\|_{X^{0,b}} \| T(u^{(0)}_{N_{0}})\|_{X^{0,b}}.
\end{equation}

Thus, using \eqref{TVEquivalenceNorm}  we conclude from \eqref{MM16} that
\begin{equation}\label{I3N2N1b}
|I_{3}(N)| \leq C \Big(\frac{N_{1}}{N_{0}}\Big)^{2} N_{2}^{s_{0} + \delta} \| u^{(1)}_{N_{1}}\|_{X^{0,b}} \| u^{(2)}_{N_{2}}\|_{X^{0,b}} \| u^{(0)}_{N_{0}}\|_{X^{0,b}}.
\end{equation}

Now, we can interpolate \eqref{I3N2N1a} and \eqref{I3N2N1b} as in \eqref{InterpolationBetweenI12andI22}, and use  Lemma \eqref{InterpolationLemma} to obtain
\begin{equation}\label{K3}
 |I_{3}(N)| \leq C \Big(\frac{N_{1}}{N_{0}}\Big)^{2} N_{2}^{s'} \| u^{(1)}_{N_{1}}\|_{X^{0,b'}} \| u^{(0)}_{N_{0}}\|_{X^{0,b'}} \| u^{(2)}_{N_{2}}\|_{X^{0,b'}}
\end{equation}
where we have $ s_{0} < s'$ and $b' < \frac{1}{2}$. 

Finally,  combining the estimates obtained for $I_{j}(N)$ $(j =1,2,3)$, in
 \eqref{K1}, \eqref{K2} and \eqref{K3}, we obtain the  required estimate  \eqref{SecondBoundForI1} as follows
\[
|I(N)| \leq |I_{1}(N)| + |I_{2}(N)| + |I_{3}(N)|\leq
 C \Big(\frac{N_{1}}{N_{0}}\Big)^{2} N_{2}^{s'} \| u^{(1)}_{N_{1}}\|_{X^{0,b'}} \| u^{(0)}_{N_{0}}\|_{X^{0,b'}} \| u^{(2)}_{N_{2}}\|_{X^{0,b'}},
\]
where $s_{0} < s'$ and  $b' < \frac{1}{2}$.
\end{proof}

Now, we come back to estimate the terms $\Sigma_{1} $ and $\Sigma_{2}$  using Lemma \ref{BoundForI1}.

\noindent{\bf Estimate  for $\Sigma_{1} $.}
 We saw in   Lemma \ref{BoundForI1} that, for a fixed $s>s_0$, one can find $s'$ with $s_{0} < s' <  s$ such that \eqref{FirstBoundForI1} holds true. Hence, one has
\begin{equation}\label{MM17}
\Sigma_{1} = \sum_{ N: N_{0} \leq C N_{1}, N_{2} \leq N_{1} } |I(N)|  \leq C  \sum_{ N: N_{0} \leq C N_{1} } N_{2}^{s'}  \|u^{(0)}_{N_{0}}\|_{X^{0, b_{1}}} \|u^{(1)}_{N_{1}}\|_{X^{0, b_{1}}} \|u^{(2)}_{N_{2}}\|_{X^{0, b_{1}}} .
\end{equation}
 Using the norm equivalence \eqref{XsbDyadicNormEquivalence}, we obtain from \eqref{MM17} that
\begin{equation}\label{MM18}
\begin{split}
\Sigma_{1} &\leq  C \sum_{ N: N_{0} \leq C N_{1} }  \Big(\frac{N_{0}}{N_{1}}\Big)^{s}   N_{2}^{s'- s} \|u^{(0)}_{N_{0}}\|_{X^{-s, b_{1}}}  \|u^{(1)}_{N_{1}}\|_{X^{s, b_{1}}}  \|u^{(2)}_{N_{2}}\|_{X^{s, b_{1}}}
\\
&=C \sum_{ N_{0}, N_{1}: N_{0} \leq C N_{1} } \Big(\frac{N_{0}}{N_{1}}\Big)^{s} \|u^{(0)}_{N_{0}}\|_{X^{-s, b_{1}}}  \|u^{(1)}_{N_{1}}\|_{X^{s, b_{1}}} \left(\sum_{N_{2}} \|u^{(2)}_{N_{2}}\|_{X^{s, b_{1}}}  N_{2}^{s'- s} \right).
\end{split}
\end{equation}

 Now, using Cauchy-Schwarz inequality in \eqref{MM18} and the norm  equivalence we find
\begin{equation}\label{MM19}
\Sigma_{1} \leq C  \|u_{2}\|_{X^{s, b_{1}}} \sum_{ N_{0}, N_{1}: N_{0} \leq C N_{1} }    \Big(\frac{N_{0}}{N_{1}}\Big)^{s} \|u^{(0)}_{N_{0}}\|_{X^{-s, b_{1}}}  \|u^{(1)}_{N_{1}}\|_{X^{s, b_{1}}}
\end{equation}

Thus, using \eqref{DyadicInequality} in  Lemma \ref{DyadicSummation}  about dyadic summations with $N = N_{0}$ and $N' = N_{1}$ in \eqref{MM19} we conclude that
\begin{equation}\label{Sigma1}
\Sigma_{1} \leq C  \|u_{2}\|_{X^{s, b_{1}}}  \|u_{1}\|_{X^{s, b_{1}}} \|u_{0}\|_{X^{-s, b_{1}}}.
\end{equation}

\noindent{\bf Estimate  for $\Sigma_{2} $.}
We use the estimate \eqref{SecondBoundForI1} and  the norm equivalence \eqref{XsbDyadicNormEquivalence}, to get
\begin{equation}\label{m-f10}
\begin{split}
 \Sigma_{2} = \sum_{ N: N_{0} > C N_{1}, N_{2} \leq N_{1} } |I(N)| & \leq C  \sum_{ N: N_{0} > C N_{1} } \Big(\frac{N_{1}}{N_{0}}\Big)^{2} N_{2}^{s'}  \|u^{(0)}_{N_{0}}\|_{X^{0, b'}} \|u^{(1)}_{N_{1}}\|_{X^{0, b'}} \|u^{(2)}_{N_{2}}\|_{X^{0, b'}}  \\
 & \leq C \sum_{ N: N_{0} > C N_{1} }  \Big(\frac{N_{1}}{N_{0}}\Big)^{2 - s}   N_{2}^{s'- s} \|u^{(0)}_{N_{0}}\|_{X^{-s, b'}}  \|u^{(1)}_{N_{1}}\|_{X^{s, b'}}  \|u^{(2)}_{N_{2}}\|_{X^{s, b'}}\\
 & = \small{ C \sum_{ N_{0}, N_{1}: N_{0} > C N_{1} }   \Big(\frac{N_{1}}{N_{0}} \Big)^{2 -s} \|u^{(0)}_{N_{0}}\|_{X^{-s, b'}}  \|u^{(1)}_{N_{1}}\|_{X^{s, b'}} \left(\sum_{N_{2}} \|u^{(2)}_{N_{2}}\|_{X^{s, b'}}  N_{2}^{s'- s} \right).}\\
\end{split}
\end{equation}

Applying the Cauchy-Schwarz inequality, we obtain from \eqref{m-f10}
\[
\Sigma_{2} \leq C  \|u_{2}\|_{X^{s, b'}} \sum_{ N_{0}, N_{1}: N_{0} > C N_{1} } \Big(\frac{N_{1}}{N_{0}} \Big)^{2-s} \|u^{(0)}_{N_{0}}\|_{X^{-s, b'}}  \|u^{(1)}_{N_{1}}\|_{X^{s, b'}}.
\]

Finally, using Lemma \ref{DyadicSummation}, similarly to \eqref{Sigma1} we conclude that
\begin{equation}\label{Sigma2}
\Sigma_{2} \leq C  \|u_{2}\|_{X^{s, b'}}  \|u_{1}\|_{X^{s, b'}} \|u_{0}\|_{X^{-s, b'}}.
\end{equation}

\noindent{\bf Estimate  for $\Sigma_{3} $ and $\Sigma_{4} $.}
Using symmetry, analogously to  $\Sigma_{1} $ and  $\Sigma_{2} $, one can obtain similar estimates for the terms $\Sigma_{3}$ and $\Sigma_{4}$.

Gathering estimates for  $\Sigma_{1} $,  $\Sigma_{2} $, $\Sigma_{3} $ and $\Sigma_{4} $ in \eqref{IMajorationSigma1234} we conclude the proof of the proposition.
\end{proof}

Now, we move to prove the  second crucial bilinear estimate stated in Proposition  \ref{SecondBilinearEstimate}.

\noindent
\begin{proof}[Proof of Proposition \ref{SecondBilinearEstimate}] By a duality argument, to prove \eqref{SecondBilinearEstimateInequality}, it  suffices to establish the following inequality
\begin{equation}\label{JDualityInequalityEquivalence}
|\int_{\mathbb{R} \times M} \overline{u_{1}} u_{2} u_{0}| \leq C \|u_{1}\|_{X^{s, b_{2}}} \|u_{2}\|_{X^{s, b_{2}}} \|u_{0}\|_{X^{-s, b}}.
\end{equation}
for all $u_{0} \in X^{-s, b}$. As in the proof of \eqref{FirstBilinearEstimateInequality} we use the dyadic decompositions $u_{j} = \sum_{j} u_{j N_{j}}$ ($j = 0, 1, 2$) in the left hand side of \eqref{JDualityInequalityEquivalence}. 

Let 
\begin{equation}\label{J-1}
J := \int_{\mathbb{R} \times M} \overline{u_{1}} u_{2} u_{0} dx dt, 
\end{equation}
and use triangle inequality to get
\begin{equation}\label{JSplit}
|J| \leq \sum_{N_{0}, N_{1}, N_{2}}  |\int_{\mathbb{R} \times M} \overline{u_{1 N_{1}}} u_{2N_{2}} u_{0N_{0}} |.
\end{equation}
 Observe that the summation in $(\ref{JSplit})$ is taken over all triples of dyadic numbers. For abbreviation, let  $N = (N_{0}, N_{1}, N_{2})$  and
\begin{equation}\label{MM33}
J(N) := \int_{\mathbb{R} \times M} \overline{u_{1 N_{1}}} u_{2N_{2}} u_{0N_{0}} dx dt.
\end{equation}
With these notations we write  $ \sum_{N} |J(N)|$ in the following manner
\begin{equation}\label{JFirstSlip}
 \sum_{N} |J(N)| \leq \sum_{N: N_{2} \leq N_{1}} |J(N)| + \sum_{N: N_{1} < N_{2}} |J(N)|.
\end{equation}

Here too, we split the sums in four  frequency regimes as we did  in the proof of \eqref{FirstBilinearEstimateInequality}. More precisely, we write
\begin{equation}\label{JSummationN2LeqN1}
 \sum_{N: N_{2} \leq N_{1}} |J(N)| \leq \sum_{N: N_{2} \leq N_{1}, N_{0} \leq C N_{1}}|J(N)| + \sum_{N: N_{2} \leq N_{1}, N_{0} > C N_{1}}|J(N)| =: \widetilde{\Sigma_{1}} + \widetilde{\Sigma_{2}},
\end{equation}
and
\begin{equation}\label{JSummationN1LeqN2}
  \sum_{N:  N_{1} < N_{2} } |J(N)| \leq \sum_{N: N_{1} < N_{2}, N_{0} \leq C N_{2}}|J(N)| + \sum_{N:  N_{1} < N_{2}, N_{0} > C N_{2}}|J(N)|=: \widetilde{\Sigma_{3}} + \widetilde{\Sigma_{4}}.
\end{equation}

Therefore, combining  \eqref{JSplit}, \eqref{JFirstSlip}, \eqref{JSummationN2LeqN1} and  \eqref{JSummationN1LeqN2} we arrive at
\begin{equation}\label{JMajorationSigma1234}
|J| \leq \widetilde{\Sigma_{1}} + \widetilde{\Sigma_{2}} + \widetilde{\Sigma_{3}} + \widetilde{\Sigma_{4}}.
\end{equation}

In this way, we reduced the proof to estimating the each term $\widetilde{\Sigma_{j}}$ ($j=1,2,3,4$). To simplify the exposition, let $u_{N_{j}}^{(j)}:= u_{jN_{j}}$ ($j = 0, 2$) and $u_{N_{1}}^{(1)}:= \overline{u_{1N_{1}}}$. We start estimating $\widetilde{\Sigma_{1}}$.

\noindent{\textbf{Estimate for the Term $\widetilde{\Sigma_{1}}$}.} In this case we have $N_{2} \leq N_{1}$; $N_{0} \leq C N_{1}.$
Consider the expression for  $J(N)$ in \eqref{MM33}. As was done to get \eqref{I1}, we use the  H\"older's inequality followed by the property  $(ii)$ in Proposition \ref{BasicXsb}, to obtain 
\begin{equation}\label{JN1}
   |J(N)| \leq C N_{2}^{3/2} \|u^{(0)}_{N_{0}}\|_{X^{0, \frac{1}{6}}}  \|u^{(1)}_{N_{1}}\|_{X^{0, \frac{1}{6}}} \|u^{(2)}_{N_{2}}\|_{X^{0, \frac{1}{6}}}.
\end{equation}

Next, as was done to  obtain  $(\ref{I2})$, we use the Cauchy-Schwarz inequality and Lemma \ref{LemaB} to find
\begin{equation}\label{JN2}
 |J(N)| \leq C N_{2}^{s_{0} + \delta} \|u^{(0)}_{N_{0}}\|_{X^{0, b}}  \|u^{(1)}_{N_{1}}\|_{X^{0,b}} \|u^{(2)}_{N_{2}}\|_{X^{0,b}}.
\end{equation}

 Now, decomposing each function $u_{N_{j}}^{(j)}$ in $(\ref{MM33})$ with respect to the time variable, we can consider
 \begin{equation}\label{J_NL}
 J(N) = \sum_{L} J(N,L),
 \end{equation}
 where
\begin{equation}\label{J_NL2}
\qquad J(N, L) := \int_{\mathbb{R} \times M} u^{(0)}_{N_{0}L_{0}} u_{N_{1}L_{1}}^{(1)} u_{N_{2}L_{2}}^{(2)} dx dt,
\end{equation}
 and the sum is taken over all dyadic integers $L=(L_0, L_1, L_2)$. 
 
 Observe that, the estimates  $(\ref{JN1})$ and $(\ref{JN2})$ also hold if one replaces $u^{(j)}_{N_{j}}$ by $u^{(j)}_{N_{j}L_{j}}$.  Now, using the norm equivalences $ \|u_{L_{j}}\|_{X^{0,b}} \simeq L_{j}^{b} \|u_{L_{j}}\|_{L^{2}(\mathbb{R} \times M)} $, we obtain from $(\ref{JN1})$ and $(\ref{JN2})$ with $u^{(j)}_{N_{j}L_{j}}$ replacing  $u^{(j)}_{N_{j}}$ that
\begin{equation}\label{JNL1}
   |J(N, L)| \leq C N_{2}^{3/2} (L_{0}L_{1}L_{2})^{\frac{1}{6}} \prod_{j = 0}^{2}\|u^{(j)}_{N_{j}L_{j}}\|_{L^{2}(\mathbb{R} \times M)},
\end{equation}
and
\begin{equation}\label{JNL2}
 |J(N, L)| \leq C N_{2}^{s_{0} + \delta} (L_{0}L_{1}L_{2})^{b} \prod_{j = 0}^{2}\|u^{(j)}_{N_{j}L_{j}}\|_{L^{2}(\mathbb{R} \times M)}.
\end{equation}

Interpolating the estimates  \eqref{JNL1} and \eqref{JNL2}, we obtain for  $0 < \theta < 1$
\begin{equation}\label{JNL1andJNL2interpolation1}
 |J(N, L)| \leq C N_{2}^{\frac{3 \theta}{2}  + (1 - \theta)(s_{0} + \delta)} (L_{0} L_{1} L_{2})^{\frac{\theta}{6} + (1- \theta )b} \prod_{j = 0}^{2}\|u^{(j)}_{N_{j}L_{j}}\|_{L^{2}(\mathbb{R} \times M)}.
\end{equation}
Consequently, using  the Lemma \ref{InterpolationLemma}, we get
\begin{equation}\label{JNL1andJNL2interpolation2}
 |J(N, L)| \leq C N_{2}^{s'} (L_{0}L_{1}L_{2})^{b'} \prod_{j = 0}^{2}\|u^{(j)}_{N_{j}L_{j}}\|_{L^{2}(\mathbb{R} \times M)},
\end{equation}
 where $s_{0} < s'$ and $b' < \frac{1}{2}$. 
 
 Hence, summing over all dyadic triples $L =(L_{0},L_{1},L_{2})$   and choosing $b_{1}$ such that $b' < b_{1} < \frac{1}{2}$,  we get from  \eqref{JNL1andJNL2interpolation2} that
\begin{equation}
|J(N)| \leq \sum_{L} |J(N,L)| \leq C N_{2}^{s'} \sum_{L}  (L_{0}L_{1}L_{2})^{b'} \prod_{j = 0}^{2}\|u^{(j)}_{N_{j}L_{j}}\|_{L^{2}(\mathbb{R} \times M)}.
\end{equation}

Now, using the norm equivalence \eqref{X0bNormEquivalenceA},  we get
\begin{equation}
\begin{split}
|J(N)|  &= C N_{2}^{s'} \sum_{L}  (L_{0}L_{1}L_{2})^{b'-b_{1}} \|u^{(0)}_{N_{0}L_{0}}\|_{X^{0,b_{1}}} \prod_{j= 1}^{2}\|u^{(j)}_{N_{j}L_{j}}\|_{X^{0,b_{1}}} \\
 &= C N_{2}^{s'} \sum_{L_{1},L_{2}}  (L_{1}L_{2})^{b'-b_{1}}  \prod_{j= 1}^{2}\|u^{(j)}_{N_{j}L_{j}}\|_{X^{0,b_{1}}} \left(\sum_{L_{0}} L_{0}^{b'-b_{1}} \|u^{(0)}_{N_{0}L_{0}}\|_{X^{0,b_{1}}}\right).
\end{split}
\end{equation}

Therefore, using the Cauchy-Schwarz inequality successively in the summations involving $L_{0}, L_{1}, L_{2}$, applying the norm equivalences \eqref{X0bNormEquivalenceA}, we obtain that for $s_{0} < s' $ there are numbers $b_{1} < \frac{1}{2}$ such that
\begin{equation}\label{JNfinal}
|J(N)| \leq C N_{2}^{s'} \|u^{(0)}_{N_{0}}\|_{X^{0, b_{1}}}  \prod_{j = 1}^{2}\|u^{(j)}_{N_{j}}\|_{X^{0, b_{1}}}   .
\end{equation}

Now,  summing according the regime $N_{2} \leq N_{1}$, $N_{0} \leq C N_{1}$, with $ s' <  s$, we have from \eqref{JNfinal} that
\begin{equation}\label{Case1bFinal}
\begin{split}
  \widetilde{\Sigma_{1}}  &\leq C   \sum_{ N: N_{0} \leq C N_{1} }   \frac{N_{0}^{s}}{N_{1}^{s}}   N_{2}^{s' -s} \|u^{(0)}_{N_{0}}\|_{X^{-s, b_{1}}}  \|u^{(1)}_{N_{1}}\|_{X^{s, b_{1}}}  \|u^{(2)}_{N_{2}}\|_{X^{s,b_{1}}} \\
  & \leq C \sum_{ N: N_{0} \leq C N_{1} }   \Big(\frac{N_{0}}{N_{1}} \Big)^{s}   N_{2}^{s' -s} \|u^{(0)}_{N_{0}}\|_{X^{-s, b_{1}}}  \|u^{(1)}_{N_{1}}\|_{X^{s, b_{1}}}  \|u^{(2)}_{N_{2}}\|_{X^{s,b_{1}}}.
\end{split}
\end{equation}

Thus, applying the Cauchy-Schwarz inequality in the summation in $N_{2}$, we obtain
\begin{equation}\label{MMM1}
\begin{split}
\widetilde{\Sigma_{1}} &\leq C \sum_{ N_{0},N_{1}: N_{0} \leq C N_{1} }   \Big(\frac{N_{0}}{N_{1}}\Big)^{s}  \|u^{(0)}_{N_{0}}\|_{X^{-s, b_{1}}}  \|u^{(1)}_{N_{1}}\|_{X^{s, b_{1}}}  \left(\sum_{N_{2}} N_{2}^{s'-s}  \|u^{(2)}_{N_{2}}\|_{X^{s,b_{1}}}\right)
\\
&\leq C  \|u_{2}\|_{X^{s,b_{1}}} \sum_{ N_{0},N_{1}: N_{0} \leq C N_{1} }   \Big(\frac{N_{0}}{N_{1}} \Big)^{s}  \|u^{(0)}_{N_{0}}\|_{X^{-s, b_{1}}}  \|u^{(1)}_{N_{1}}\|_{X^{s, b_{1}}}.
\end{split}
\end{equation}

Finally, in view of  Lemma \ref{DyadicSummation}, we obtain from \eqref{MMM1} that
\[
\widetilde{\Sigma_{1}} \leq C \|u_{0}\|_{X^{-s, b_{1}}}  \|u_{1}\|_{X^{s, b_{1}}}   \|u_{2}\|_{X^{s,b_{1}}}.
\]

\noindent{\textbf{Estimate for the Term $\widetilde{\Sigma_{2}}$}.} In this case we have $N_{2} \leq N_{1}$; $N_{0} > C N_{1}$.

In the same way we did in the proof of \eqref{FirstBilinearEstimateInequality}, (see Lemma \ref{BoundForI1}) we split the integral $J$ into three terms and analyze each one of them separately. More precisely, we write
\[
J(N) = J_{1}(N) + J_{2}(N) + J_{3}(N),
\]
where
\begin{equation}\label{SplitIntegralJ}
\begin{cases}
 J_{1}(N) := - \frac{1}{N_{0}^{2}} \int_{\mathbb{R} \times M}  T(u^{(0)}_{N_{0}}) u^{(1)}_{N_{1}} \Delta_{g}( u^{(2)}_{N_{2}})  dx dt ,\\
 J_{2}(N) := - \frac{1}{N_{0}^{2}} \int_{\mathbb{R} \times M}  T(u^{(0)}_{N_{0}}) u^{(2)}_{N_{2}} \Delta_{g}( u^{(1)}_{N_{1}})  dx dt , \\
 J_{3}(N) := - \frac{2}{N_{0}^{2}} \int_{\mathbb{R} \times M}  T(u^{(0)}_{N_{0}}) (\nabla u^{(1)}_{N_{1}},\nabla u^{(2)}_{N_{2}})_{g}  dx dt.
\end{cases}
\end{equation}

We can obtain the following estimates for $J_{k}(N)$.
\begin{equation}\label{J1a}
 |J_{k}(N)| \leq \Big(\frac{N_{1}}{N_{0}}\Big)^{2} N_{2}^{\frac{3}{2}} \prod_{j=0}^{2}\|u^{(j)}_{N_{j}}\|_{X^{0, 1/6}},
\end{equation}
and
\begin{equation}\label{J12a}
 |J_{k}(N)| \leq \Big(\frac{N_{1}}{N_{0}}\Big)^{2} N_{2}^{s_{0} + \delta} \prod_{j=0}^{2}\|u^{(j)}_{N_{j}}\|_{X^{0,b}},
\end{equation}
 where $k = 1,2,3$. We can use the same estimates and considerations as we did to estimate the terms $I_{k}$, $k = 1,2,3$, in Lemma $\ref{BoundForI1}$. More explicitly:
\begin{enumerate}
\item[$\bullet$] For $k = 1$ we use the same arguments that were used in the estimates $(\ref{I1N2N1a})$ and $(\ref{I1N2N1b})$.
\item[$\bullet$]  For $k = 2$ we use the same arguments that were used in the estimates  $(\ref{I2N2N1a})$ and $(\ref{I2N2N1b})$.
\item[$\bullet$] For $k = 3$ we use the same arguments that were used in the estimates  $(\ref{I3N2N1a})$ and $(\ref{I3N2N1b})$.
\end{enumerate}

Hence, interpolating $(\ref{J1a})$ and $(\ref{J12a})$, for each $k = 1, 2,3$ we see that for $s_{0} < s'$, and  $b' < \frac{1}{2}$ one has
\begin{equation}\label{MMM2}
|J(\underline{N})| \leq \sum_{k =1}^{3} |J_{k}(\underline{N})| \leq C  \Big(\frac{N_{1}}{N_{0}} \Big)^{2}  N_{2}^{s'-s} \frac{N_{0}^{s}}{N_{1}^{s}} \|u^{(0)}_{N_{0}}\|_{X^{-s, b'}} \|u^{(1)}_{N_{1}}\|_{X^{s, b'}}  \|u^{(2)}_{N_{2}}\|_{X^{s,b'}}.
\end{equation}
Hence, summing in  $N = (N_{0}, N_{1}, N_{2})$, we get
\begin{equation}\label{Case2bFinal}
\begin{split}
\widetilde{\Sigma_{2}}  &\leq C \sum_{N: N_{0} > C N_{1}} \Big(\frac{N_{1}}{N_{0}}\Big)^{2 - s} N_{2}^{s' - s} \|u^{(0)}_{N_{0}}\|_{X^{-s, b'}} \|u^{(1)}_{N_{1}}\|_{X^{s, b'}}  \|u^{(2)}_{N_{2}}\|_{X^{s,b'}}    \\
 &\leq C \|u_{0}\|_{X^{-s, b'}}  \|u_{1}\|_{X^{s, b'}}   \|u_{2}\|_{X^{s,b'}}. \\
\end{split}
\end{equation}

\noindent{\bf Estimate for the terms $\widetilde{\Sigma_{3}}$ and $\widetilde{\Sigma_{4}}$.} By a symmetry argument we can prove the same 
estimates for $\widetilde{\Sigma_{3}}:N_{1} < N_{2}$; $N_{0} \leq C N_{2}.$
and $\widetilde{\Sigma_{4}}:N_{1} < N_{2}$; $N_{0} > C N_{2}.$

Finally, collecting the estimates established in $(\ref{Case1bFinal})$, $(\ref{Case2bFinal})$ and for the summations $\widetilde{\Sigma_{j}}$ $(j = 3, 4)$ we obtain the required estimate
\begin{equation*}
\begin{split}
|J| &\leq \sum_{N}|J(N)| \leq \widetilde{\Sigma_{1}} + \widetilde{\Sigma_{2}}  + \widetilde{\Sigma_{3}}  + \widetilde{\Sigma_{4}}\\
&\leq C \|u_{0}\|_{X^{-s, b}} \|u_{1}\|_{X^{s, b_{2}}}  \|u_{2}\|_{X^{s,b_{2}}},
\end{split}
\end{equation*}
where   $b_{2} > b_{1}, b'$ are chosen in a suitable manner.
\end{proof}

\section{Proof the main result}
In this section we use the estimates obtained in the previous section to prove the local well-posedness stated in Theorem \ref{local-Th} for the IVP \eqref{SDS1} . 

\noindent
\begin{proof}[Proof of Theorem \ref{local-Th}] Let  $s>\frac23$ and  $u_{0} \in H^{s}(M)$. Applying  Duhamel's formula, we can rewrite the IVP \eqref{SDS1} in  the following equivalent  integral equation  (with Dirichlet or Neumann conditions)
\begin{equation}\label{SDS-0}
 u(t) = e^{it \Delta} u_{0} - i \int_{0}^{t} e^{i (t - \tau) \Delta}  Q(u(\tau), \overline{u}(\tau)) d \tau,
\end{equation}
where $e^{it \Delta}$ denotes the evolution of the linear Schr\"odinger equation defined using Dirichlet or Neumann spectral resolution and $  Q(u, \overline{u}):=\alpha  u^{2} + \beta  \overline{u}^2 + \gamma |u|^2$. 

We define an application 
\begin{equation}\label{SDS2}
 \Phi(u)(t) := e^{it \Delta} u_{0} - i \int_{0}^{t} e^{i (t - \tau) \Delta}  Q(u(\tau), \overline{u}(\tau)) d \tau,
\end{equation}
 and  use the contraction mapping principle to find a fixed point $u$ that solves the equation \eqref{SDS-0}.
 
 For this,  let $T>0$ and $R>0$ to be chosen later and consider a ball 
\[
B_{T}^R := \{ u \in X^{s,b}_{T} ; \|u\|_{X_{T}^{s, b}} \leq R \}
\]
 in the space $ X^{s,b}_{T} $. We will show that for  sufficiently small  $T>0$ and an appropriate positive constant $R > 0$, the application $\Phi$ defined in \eqref{SDS2} is a contraction map.  In fact,  applying the linear estimates \eqref{XsbLinearEstimateA} and \eqref{XsbLinearEstimateB} from Proposition \ref{LinearEstimates1} in  \eqref{SDS2}, we obtain for $T\leq 1$
\begin{equation}\label{Pf-1}
\|\Phi (u)\|_{X_{T}^{s,b}} \leq c_{0} \|u_{0}\|_{H^{s}(M)} + c_{1} T^{1 - b - b'} \|Q(u, \overline{u}) \|_{X_{T}^{s, -b'}}.
\end{equation}

Using the bilinear estimates  \eqref{FirstBilinearEstimateInequality} and \eqref{SecondBilinearEstimateInequality}, we get from \eqref{Pf-1}
\begin{equation}\label{Pf-2}
\|\Phi (u)\|_{X_{T}^{s,b}} \leq c_{0} \|u_{0}\|_{H^{s}(M)} +  c_{1}T^{1 - b - b'} \|u\|_{X_{T}^{s,b} }^{2}.
\end{equation}

Set   $\theta_{1}:= 1-b-b' > 0$ and  $R :=  2 c_{0} \|u_{0}\|_{H^{s}(M)} $. Therefore, for $u \in B_{T}^R$ the estimate \eqref{Pf-2} yields
\begin{equation}\label{Pf-3}
\|\Phi(u)\|_{X_{T}^{s,b}} \leq \frac{R}{2} + c_{1} T^{\theta_{1}} {R}^{2}. 
\end{equation}

Hence, for a suitable $0 < T\leq 1$ such that $c_1T^{\theta_1}R<\frac12$, one can conclude that  $\Phi$ maps $B_{T}^R$ onto itself. 

Let $u, \tilde{u}\in B_T^R$ be  solutions with the same initial data $u_0$. In an analogous manner as we did above, it is easy to get
\begin{equation}\label{Pf-4}
\|\Phi (u) -\Phi (\tilde{u}) \|_{X_{T}^{s,b}} \leq C  T^{1 - b - b'} \|Q(u, \tilde{u})- Q(\tilde{u}, \overline{\tilde{u}}) \|_{X_{T}^{s,-b'}}.
\end{equation}

Observe that, we can write 
 \begin{equation}\label{Decom}
\begin{cases}
\overline{u}^2 - \overline{\tilde{u}}^2 &= \overline{(u - \tilde{u})} \overline{u} + \overline{(u - \tilde{u})} \overline{\tilde{u}}, \\
u^2 - \tilde{u}^2 &= ( u - \tilde{u}) u + (u - \tilde{u}) \tilde{u}, \\
|u|^{2} - |\tilde{u}|^{2} &= u (\overline{u - \tilde{u}}) + (u - \tilde{u}) \overline{\tilde{u}}.
\end{cases}
\end{equation}

Now, using \eqref{Decom} in \eqref{Pf-4}  and then applying the  bilinear estimates  \eqref{FirstBilinearEstimateInequality} and \eqref{SecondBilinearEstimateInequality}, we obtain
\begin{equation}\label{Pf-5}
\|\Phi(u) -\Phi(\tilde{u}) \|_{X_{T}^{s, b}}  
\leq C T^{1 - b- b'} (\|u \|_{X_{T}^{s,b} } + \|\tilde{u} \|_{X_{T}^{s,b}}) \|  u - \tilde{u}\|_{X_{T}^{s,b}} \leq CT^{\theta_1}R \|  u - \tilde{u}\|_{X_{T}^{s,b}}.
\end{equation}

If we choose $0<T\leq 1$ such that $\max\{c_1T^{\theta_1}R, \;CT^{\theta_1}R\}<\frac12$, it follows from \eqref{Pf-5} that $\Phi$ is a  contraction on the ball $B_T^R$. The Lipschitz property is obtained with a similar idea, so the details are omitted.
\end{proof}

\begin{obs}
Let $B := \mathbb{B}^{3}$ the unit ball in $\mathbb{R}^{3}$, denote the Laplacean in $B$ by $\Delta: = \Delta_{B}$ and consider the linear Schr\"odinger group $S(t) = e^{i t \Delta_{B}}$.  Anton in  \cite{RAMONA} considered radial data $u_{0}, v_{0}$  spectrally  localized at frequency $\Gamma, \Lambda$ respectively  to prove 
\[
\| S(t) u_{0} S(t) v_{0}\|_{L^{2}((0,1) \times B)} \leq C (\min (\Gamma, \Lambda))^{s} \|u_{0}\|_{L^{2}(B)} \|v_{0}\|_{L^{2}(B)},
\]
and
\[
\|(\nabla S(t) u_{0}) S(t) v_{0}\|_{L^{2}((0,1) \times B)} \leq C \Gamma(\min (\Gamma, \Lambda))^{s} \|u_{0}\|_{L^{2}(B)} \|v_{0}\|_{L^{2}(B)},
\]
for any $s>\frac12$.

With these estimates at hand, we can obtain the  bilinear estimates established in Propositions \ref{FirstBilinearEstimate} and \ref{SecondBilinearEstimate} for the quadratic NLS \eqref{SDS1} posed on $\mathbb{B}^{3}$ as well. Consequently, as in \cite{RAMONA}, we can also  establish a local well-posedness result for the quadratic NLS equation  \eqref{SDS1} with radial data in $H^{s}(\mathbb{B}^3)$ for  $s> 1/2$.
\end{obs}

\section{Appendix}
In this appendix, we will  prove the crucial duality argument used in \eqref{DualityInequalityEquivalence}.  The aim is to prove that there exists an isometric isomorphism $\Phi: X^{-s, -b}(\mathbb{R} \times M) \rightarrow (X^{s,b}(\mathbb{R} \times M))^{\ast}$ such that 
$ \| \Phi(f)\| = \| f \|_{X^{-s, -b}}$ . For this, we need to introduce some definitions and notations. 
To begin, let us define (for $f \in C^{\infty}_{0} (\mathbb{R} \times M)$)
\[
J^{s} f = \sum_{k} \langle \mu_{k} \rangle^{s/2} P_{k} f ,
\]
and via Fourier transform, 
\[
\Lambda_{k}^{b} f (t) = (\langle \tau + \mu_{k} \rangle^{b} \hat{f}\check{)}(t).
\]

In this manner, we can define the operator 
\[
J^{s} \Lambda^{b} f := \sum_{k} \langle \mu_{k} \rangle^{s/2} P_{k} [\Lambda_{k}^{b} f (t)].
\]
Using this definition and considering $u \in X^{s,b} \cap L^{2}_{tx}$, $v \in X^{-s,-b} \cap L^{2}_{tx}$, we have
\[
\langle J^{s} \Lambda^{b} v , J^{-s} \Lambda^{-b} u \rangle_{L^{2}_{tx}} = \langle v(t) , u(t) \rangle_{L^{2}_{tx}}. 
\]
In fact, using $L^{2}$-orthogonality and Plancherel's theorem, we obtain 
\begin{equation}
\begin{split}
\langle J^{s} \Lambda^{b} v , J^{-s} \Lambda^{-b} u \rangle_{L^{2}_{tx}} &= \sum_{k} \int_{\mathbb{R} \times M}  P_{k}[\Lambda_{k}^{b} v (t)] \overline{P_{k}[\Lambda_{k}^{-b} u (t)]} dg dt \\
&= \sum_{k} \int_{\mathbb{R} \times M}  P_{k}[\langle \tau + \mu_{k} \rangle ^{b} \hat{v} (\tau)\check{]}(t) \overline{P_{k}[\langle \tau + \mu_{k} \rangle^{-b} \hat{u}(\tau) \check{]} (t) }dg dt \\
&= \sum_{k} \int_{\mathbb{R} \times M}  P_{k}[\langle \tau + \mu_{k} \rangle ^{b} \hat{v} (\tau)] \quad{} \overline{P_{k}[\langle \tau + \mu_{k} \rangle^{-b} \hat{u}(\tau) ]  }dg d\tau \\
&= \sum_{k} \int_{\mathbb{R} \times M}  P_{k}\hat{v} (\tau) \quad{} \overline{P_{k} \hat{u}(\tau)   }dg d\tau.
\end{split}
\end{equation}

Now, using that $\widehat{P_{k} f } = P_{k} \hat{f}$, we conclude 
\begin{equation}
\begin{split}
\langle J^{s} \Lambda^{b} v , J^{-s} \Lambda^{-b} u \rangle_{L^{2}_{tx}} &= \sum_{k} \int_{\mathbb{R} \times M}  \widehat{P_{k}v} (\tau) \quad{} \overline{ \widehat{P_{k} u}(\tau)   }dg d\tau \\
& = \sum_{k} \int_{\mathbb{R} \times M}  \widehat{P_{k}v} (\tau)  \widehat{ \overline {P_{k} u}}(-\tau) dg d\tau \\
 & = \sum_{k} \int_{ M} \int_{\mathbb{R}} P_{k}v(t)    \overline {P_{k} u}(t) dg dt \\
& = \sum_{k}  \int_{\mathbb{R}} \langle P_{k}v(t)    P_{k} u(t) \rangle_{L^{2}(M)} dt \\
 & =   \langle v(t), u(t) \rangle_{L^{2}( \mathbb{R} \times M)} . 
\end{split}
\end{equation}

Now, we can state the following lemma. 
\begin{lema}
Let  $\langle \cdot, \cdot \rangle $ denotes an inner product in $L^{2}_{tx}$. 
 Let $\Phi: X^{-s, -b}(\mathbb{R} \times M) \rightarrow (X^{s,b}(\mathbb{R} \times M))^{\ast}$ be defined by 
\[
\Phi_{h} (f) = \langle  J^{s} \Lambda^{b} f, J^{-s} \Lambda^{-b}  h \rangle
\]
Then $\Phi$ is an isometric isomorphism and we have 
$\Phi_{h}(f) = \langle f, h \rangle$, whenever 
$f \in X^{s, b} \cap L^{2}_{tx}$ and $h \in X^{-s, -b} \cap L^{2}_{tx}$. 
\end{lema}
\begin{proof}
For $f \in X^{s,b}$ and $h \in X^{-s, -b}$, we have 
\begin{equation}
\begin{split}
|\Phi_{h} (f)| &= | \langle  J^{s} \Lambda^{b} f, J^{-s} \Lambda^{-b}  h \rangle| \\
&\leq \| J^{s} \Lambda^{b} f\|_{L^{2}}  \| J^{-s} \Lambda^{-b} h\|_{L^{2}} \\
&= \| f\|_{X^{s,b}} \| h\|_{X^{-s,-b}} .
\end{split}
\end{equation}
Hence $\Phi_{h} \in (X^{s,b})^{\ast}$ with $\| \Phi_{h} \| \leq \|h \|_{X^{-s, -b}}$. Moreover,
\begin{equation}
\begin{split}
\|\Phi_{h} \| & = \sup_{\| f \|_{X^{s,b}} \leq 1} |\langle  J^{s} \Lambda^{b} f, J^{-s} \Lambda^{-b}  h \rangle| \\
&= \sup_{\| \ell \|_{X^{s,b}} \leq 1} |\langle  \ell, J^{-s} \Lambda^{-b}  h \rangle| \\
&=\| J^{-s} \Lambda^{-b}  h \|_{L_{tx}^{2}} \\
&=\| h \|_{X^{-s, -b}}. 
\end{split}
\end{equation}

It remains to show that $\Phi$ is onto. Let $y$ be a bounded linear functional on $X^{s,b}$. Then 
$z = y \circ J^{-s} \Lambda^{-b} $ is a bounded linear functional on $L^{2}_{tx}$ and by the Riesz's representation theorem there exists $\tilde{h} \in L^{2}_{tx}$ with $z(\tilde{f}) = \langle \tilde{f}, \tilde{h} \rangle$ for all $\tilde{f} \in L^{2}_{tx}$. Now, note that $h: = J^{s} \Lambda^{b} \tilde{h}$ belongs to $X^{-s, -b}$ and it is easy to show that $y(f) = \Phi_{h} (f)$ for all $f \in X^{s, b}$. Finally, let $f \in X^{s, b} \cap L^{2}_{tx}$ and $h \in X^{-s, -b} \cap L^{2}_{tx}$. From the above computations we have
\[
\langle f, h \rangle = \langle J^{s} \Lambda^{b} f, J^{-s} \Lambda^{-b} h \rangle, 
\]
and the proof is completed. 
\end{proof}

\bibliographystyle{plain}

\begin{thebibliography}{9}

\bibitem{RAMONA} R. Anton, \textit{Cubic nonlinear Schr\"odinger equation on three dimensional balls with radial data},
Communications in Partial Differential Equations \textbf{33} (2008) 1862--1889.

\bibitem{BT2006} I. Bejenaru, T. Tao; \textit{Sharp well-posedness and ill-posedness results for a quadratic non-linear Schr\"odinger equation}, J. of Func. Analysis,  $\textbf{233}$ (2006) 228--259.

\bibitem{PIERREBERARD}  P. B\'erard, \textit{Spectral Geometry : Direct and Inverse Problems }, Lecture Notes in Mathematics. \textbf{1207} (1986).

\bibitem{BSS2008} M. D. Blair, H. F. Smith, C. D. Sogge; \textit{On Strichartz Estimates for  Schr\"odinger Operators in Compact Manifolds With Boundary}, Proc. of the Amer. Math. Soc., \textbf{136} (2008) 247--256. 

\bibitem{BSS2012} M. D. Blair, H. F. Smith,  C. D. Sogge; \textit{Strichartz Estimates and the Nonlinear  Schr\"odinger Equation on Manifolds with Boundary}, Math. Ann., \textbf{354}  (2012) 1397--1430.

 \bibitem{B1993} J. Bourgain, {\em  Fourier transform restriction phenomena for certain lattice subsets and applications to nonlinear evolution equation  I},  Schr\"odinger equations. Geom. Funct. Anal. {\bf 3}
 (1993) 107--156. 
\bibitem{BGT} N. Burq, P. G\'erard, N.Tzvetkov, \textit{Strichartz inequalities and the nonlinear Schr\"odinger equation on compact manifolds},
Amer. J. Math.\textbf{126}(3), (2004) 569-605.

\bibitem{BGT1} N. Burq, P. G\'erard, N.Tzvetkov, \textit{Multilinear eigenfunction estimates and global existence for the three dimensional nonlinear Schr\"odinger equations }, Ann. Scient. Ec. Norm. Sup.,(2005) 255-301.

\bibitem{BGT2} N. Burq, P. G\'erard, N.Tzvetkov, \textit{ Bilinear eigenfunction estimates and the nonlinear  Schr\"odinger  equation on surfaces}, Invent. Math. \textbf{159} (2005) 187-223.

\bibitem{JIANG} J. C. Jiang, \textit{Bilinear Strichartz estimates for Schr\"odinger operators in two-dimensional compact manifolds with Boundary and NLS},
Differential and Integral Equations, Vol.\textbf{24}, Numbers $1-2$ (2011) 83-108.

 \bibitem{PIERFELICEGERARD} P. G\'erard, V. Pierfelice, \textit{ Nonlinear Schr\"odinger equation on Four-Dimensional compact manifolds}, Bull. Soc. math. France,
 \textbf{138}(1),(2010) 119-151.



\bibitem{GINIBRE} J. Ginibre, \textit{Le probl\'eme de Cauchy pour des EDP semi-lin\'eaires p\'eriodiques en variables d'espace (d'apr\'es Bourgain)}, S\'eminaire Bourbaki, Exp. 796, Ast\'erisque $\textbf{237}$ (1996) 163-187.

\bibitem{GINIBREVELO} J, Ginibre, G. Velo, \textit{On a class of nonlinear Schr\"odinger equations. $I$ and $II$ }, J. Funct. Anal., \textbf{32} (1979), 1-32, 33-72.

\bibitem{KPV96} C. Kenig, G. Ponce, L. Vega; \textit{Quadratic forms for the 1-D semilinear Schr\"odinger equation}, Trans.
Amer. Math. Soc. \textbf{346} (1996) 3323--3353.

\bibitem{KIS08} N. Kishimoto; \textit{Local well-posedness for the Cauchy problem of the quadratic Schr\"odinger equation with nonlinearity $\overline{u}^{2}$}, Communications on Pure and  Applied Analysis, \textbf{7} (5) (2008), 1123--1143. 

\bibitem{LP2015B} F. Linares, G. Ponce; \textit{Introduction to nonlinear dispersive equations}, Second edition, Universitext, Springer, New York, (2015). 

\bibitem{NP2020} M. Nogueira, M. Panthee; \textit{On the Schr\"odinger-Debye system in compact Riemannian manifolds}, Communications on Pure and Applied Analysis,  \textbf{19}(1) (2020), 425--453. 

\bibitem{TAO2006B} T. Tao; \textit{Nonlinear dispersive equations: local and global analysis}, Regional Conference Series in Mathematics, \textbf{106}. American Mathematical Society, Providence, RI, (2006).



\end{thebibliography}

\end{document}